\documentclass[11pt]{article}
\usepackage{amsmath}
\usepackage[ruled,vlined]{algorithm2e}
\usepackage{graphicx}
\usepackage{enumerate}
\usepackage{url} 
\usepackage{textcomp}
\usepackage{mathtools}
\usepackage{amssymb}
\usepackage{amsthm}
\usepackage{setspace}
\usepackage{caption}
\usepackage{mleftright}
\usepackage{multirow}
\usepackage[utf8]{inputenc}
\usepackage[english]{babel}
\usepackage{filecontents}
\usepackage{sectsty}
\usepackage{titletoc}
\usepackage{xcolor}
\usepackage[toc,page]{appendix}
\graphicspath{ {figures/} }
\usepackage{array}
\usepackage{float}
\usepackage{scalerel,stackengine}

\DeclarePairedDelimiter\floor{\lfloor}{\rfloor}
\newtheorem{theorem}{Theorem}[section]
\newtheorem{proposition}{Proposition}[section]
\newtheorem{corollary}{Corollary}[theorem]
\newtheorem{lemma}[theorem]{Lemma}

\newcommand{\mbf}[1]{\mbox{\boldmath $#1$}}
\newcommand{\bth}{{\mbf \theta}}
\newcommand{\bUpsilon}{{\mbf \Upsilon}}
\newcommand{\bTH}{{\mbf \Theta}}

\newcommand{\bt}{\theta}

\newcommand{\benr}{\begin{eqnarray}}
\newcommand{\eenr}{\end{eqnarray}}
\newcommand{\ben}{\begin{equation}}
\newcommand{\een}{\end{equation}}
\newcommand{\benrr}{\begin{eqnarray*}}
\newcommand{\eenrr}{\end{eqnarray*}}

\numberwithin{equation}{section}
\numberwithin{theorem}{section}

\def\d{\dot}
\def\del{\delta}
\def\ep{\epsilon}
\def\g{\gamma}

\def\h{\hat}
\def\iny{\infty}

\def\nn{\nonumber}

\def\r{\ref}

\def\R{{\mathbb R}}

\def\ZZ{{\mathbb Z}}
\def\NN{{\mathbb N}}

\def\s{\sum_{t=1}^n}

\def\ti{\tilde}

\def\vp{\varphi}

\def\bth{\mbox{\boldmath$\theta$}}

\def\bTheta{\mbox{\boldmath$\Theta$}}
\def\bDelta{\mbox{\boldmath$\Delta$}}
\def\bt{\mbox{\boldmath$t$}}
\def\bs{\mbox{\boldmath$s$}}
\def\bb{\mbox{\boldmath$b$}}

\def\w{{\mbf w}}
\def\M{{\mbf M}}

\def\Q{{\mbf Q}}

\def\T{{\mbf T}}

\def\A{{\mbf A}}

\def\D{{\mbf D}}
\def\I{{\mbf I}}
\def\J{{\mbf J}}
\def\K{{\mbf K}}
\def\0{{\mbf 0}}

\def\N{{\mbf N}}

\def\M{{\mbf M}}

\newcommand{\blind}{1}

\begin{document}

\def\spacingset#1{\renewcommand{\baselinestretch}%
{#1}\small\normalsize} \spacingset{1.5}


\if1\blind
{
  \title{\bf R-estimators in GARCH models; asymptotics, applications and bootstrapping}
  
  \author{Hang Liu\footnote{Email: h.liu11@lancaster.ac.uk}\hspace{.2cm} and Kanchan Mukherjee\footnote{The author for correspondence. Email: k.mukherjee@lancaster.ac.uk} \\
    Department of Mathematics and Statistics, Lancaster University
    }
  \maketitle
} \fi

\if0\blind
{
  \bigskip
  \bigskip
  \bigskip
  \begin{center}
    {\LARGE\bf Title}
\end{center}
  \medskip
} \fi

\bigskip
\begin{abstract}
The quasi-maximum likelihood estimation is a commonly-used method for estimating GARCH parameters. However, such estimators are sensitive to outliers and their asymptotic normality is proved under the finite fourth moment assumption on the underlying error distribution. In this paper, we propose a novel class of estimators of the GARCH parameters based on ranks, called R-estimators, with the property that they are asymptotic normal under the existence of a more than second moment of the errors 
and are highly efficient. We also consider the weighted bootstrap approximation of the finite sample distributions of the R-estimators. We propose fast algorithms for computing the R-estimators and their bootstrap replicates. Both real data analysis and simulations show the superior performance of the 
proposed estimators under the normal and heavy-tailed distributions. Our extensive simulations also reveal excellent coverage rates of the weighted bootstrap approximations. In addition, we discuss empirical and simulation results of the R-estimators for the higher order GARCH models such as the GARCH~($2, 1$) and asymmetric models such as the GJR model.
\end{abstract}

\noindent%
{\it Keywords:}  R-estimation, Empirical process, Higher order GARCH models, Weighted bootstrap.

{\it JEL classification:} C13, C14, C22.  
\vfill

\newpage
\spacingset{1.5} 

\section{Introduction}
\setcounter{equation}{0} 

\subsection{Robust estimation based on ranks}

Consider observations $\{X_t; 1 \le t \le n\}$ from a financial time series with the following representation
\begin{equation*}
X_t= \sigma_{t}\epsilon_t,
\end{equation*}
where $\{\epsilon_t; t \in \ZZ\}$ are unobservable i.i.d. non-degenerate error r.v.'s with mean zero 
and unit variance and 
\begin{equation}\label{m2}
\sigma_{t}= (\omega_0+\sum_{i=1}^p \alpha_{0i}X^2_{t-i}+\sum_{j=1}^q \beta_{0j} \sigma^2_{t-j})^{1/2}, \,\,t \in \ZZ,
\end{equation}
with $\omega_0$, $\alpha_{0i}, \,\, \beta_{0j}> 0$, $\forall \, i, j$. In the literature, such models are known as the GARCH~($p, q$) model and we assume that $\{X_t; t \in \ZZ\}$ is stationary and ergodic.

Estimation of parameters based on ranks of the residuals was discussed by Koul and Ossiander~(1994) for the homoscedastic autoregressive model and Mukherjee~(2007) for the heterscedastic models. Andrews~(2012) proposed a class of  R-estimators for the GARCH model using a log-transformation of the squared observations and then minimizing a rank-based residual dispersion function. However, such square-transformation may lead to loss of information under an asymmetric innovation distribution which is undesirable. This motivates us to define R-estimators for the GARCH model that uses the data directly without requiring such transformation. As shown in the motivating example of Section~\ref{sec.motivate}, our proposed R-estimators can be more efficient than those of Andrews (2012) when asymmetry is introduced for the innovation distribution while retaining efficiency for symmetric innovation distributions. Similar to the linear regression and autoregressive models, the asymptotic normality of R-estimators are derived under smoothness conditions of the innovation probability density function instead of that of the logged and squared innovation as in Andrews (2012).  

As expected, the proposed class of R-estimators turns out to be robust and relatively efficient. Unlike the commonly-used quasi-maximum likelihood estimator (QMLE) which is asymptotically normal under the finite fourth moment assumption of the error distribution, the R-estimators turn out to be asymptotically normal under the assumption of only a finite 
$2+\delta$-th moment for some $\delta>0$. The efficiency property of the R-estimators is further confirmed based on the simulated data from the GARCH~($1, 1$) model and the higher order GARCH~($2, 1$) model to fill some void in the literature since the computation and empirical analysis for the higher order GARCH models are not considered widely. Analysis of real data shows that the numerical values of R-estimates can be different from the QMLE and the subsequent analysis of the GARCH residuals shows that such difference may be attributed to
the infinite fourth moment of the innovation distribution, which leads to the failure of the QMLE.


Robust estimation of the GARCH parameters has been studied extensively in the literature although the attention has been focused exclusively on the class of M-estimators except in Andrews (2012). See, for example, Berkes and Horváth (2004), Mukherjee (2008), Francq et al.~(2011), Fan et al.~(2014) and Zhu and Ling~(2011) and the references in those papers. One conspicuous issue with previous studies is related to the estimation of an identifiable scale parameter that leads to often more than one stage of estimation. As pointed out by Fan et al.~(2014, Section 7.2), such estimation is important for comparing the bias performance. Some simulation study by Fan et al.(2014, Section 7.2) to compare M-estimators with the R-estimators proposed by Andrews (2012) revealed that these two classes of estimators have almost indistinguishable asymptotic performance while the rank-based estimators are slightly better and this provides another motivation for considering R-estimators. However, one problem unaddressed in Andrews (2012) was that the scale was not estimated. In Section 
\ref{subsec.cvp} of this paper, we propose a simple consistent estimate of the scale based on R-estimators and the general principle can be applied for M-estimation as well.

Since the proposed class of the R-estimators are shown to converge to normal distributions, of which the covariance matrices do not have explicit forms, we employ a bootstrap method to approximate the distributions of the R-estimators. 
Chatterjee and Bose~(2005) used the weighted bootstrap method for an estimator defined by smooth estimating equations. We consider weighted bootstrap of R-estimator where the equations are non-smooth functions because the residual ranks are integer-valued and non-smooth.
Our extensive simulation study provides evidence that the weighted bootstrap has good coverage rates even under heavy-tailed innovation distribution and with moderate sample size.

Finally, we use the R-estimators for estimating parameters of the GJR~($p, q$) model proposed by Glosten et al.~(1993), which is used to estimate the asymmetry effect of financial time series. Simulation results demonstrate good performance of the R-estimators for the GJR model.

The main contributions of the paper are threefold. First, a new class of robust and efficient estimators for the GARCH model parameters is proposed. Second, the asymptotic distributions of the proposed estimators are derived based on weak assumption on the error moments. Third, weighted bootstrap approximations of the distribution of the R-estimators are investigated through extensive simulations. In particular, we propose algorithms for computing the R-estimators and the bootstrap replicates, which are computational friendly and easy to implement. 

\subsection{A motivating example}\label{sec.motivate}

\begin{figure}[!htbp]
\caption{Boxplots of the QMLE, BA's and our proposed R-estimators  (van der Waerden) under the skew normal (upper panel) and standard normal (lower panel) innovation densities. In each panel, the MSE ratios of the QMLE with respect to other estimators are reported. The horizontal line represents the actual parameter value.\vspace{-5mm}}
\begin{center}
\begin{tabular}{c}
\includegraphics[width=0.9\textwidth, height=0.5\textwidth]{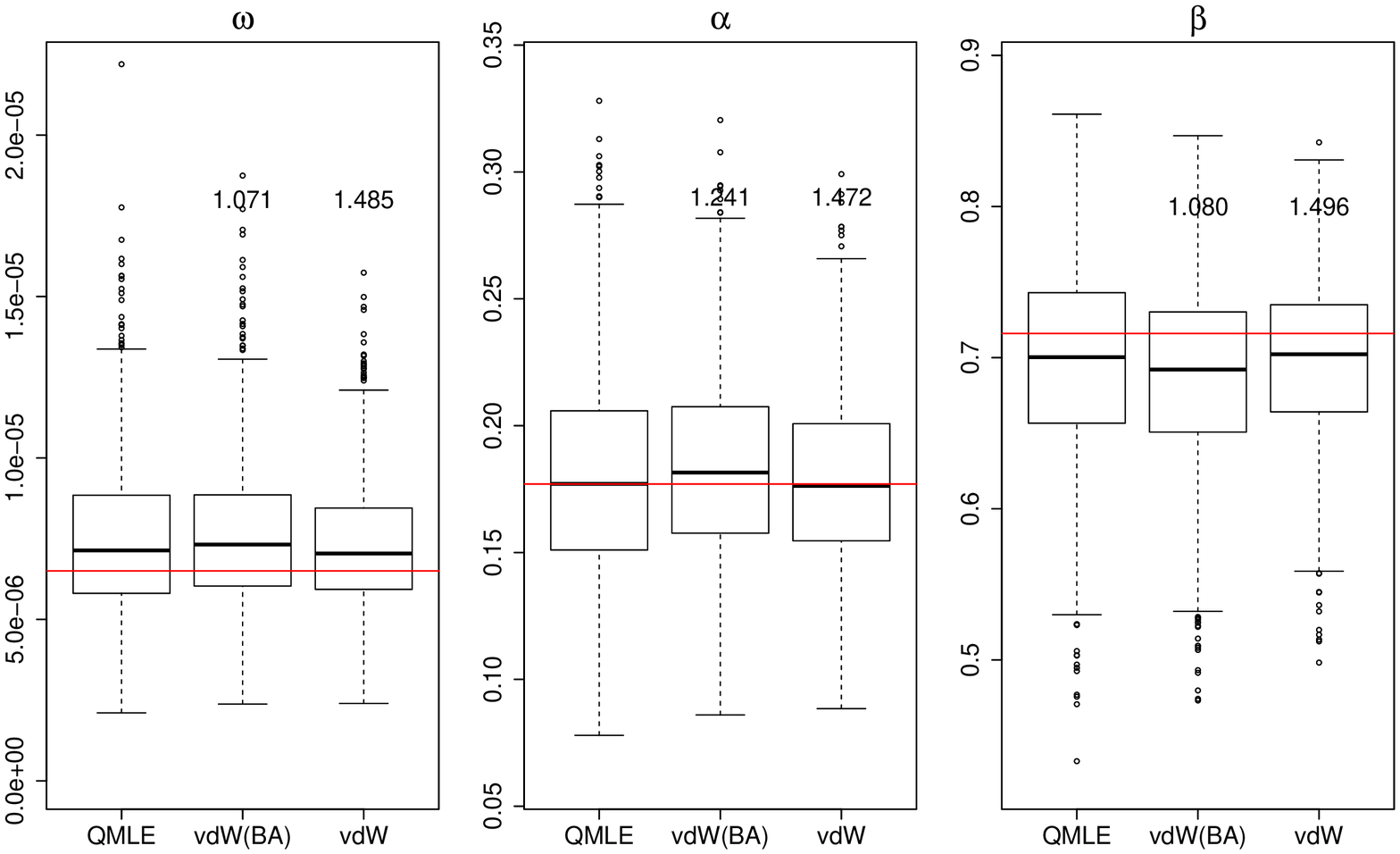} \\
\includegraphics[width=0.9\textwidth, height=0.5\textwidth]{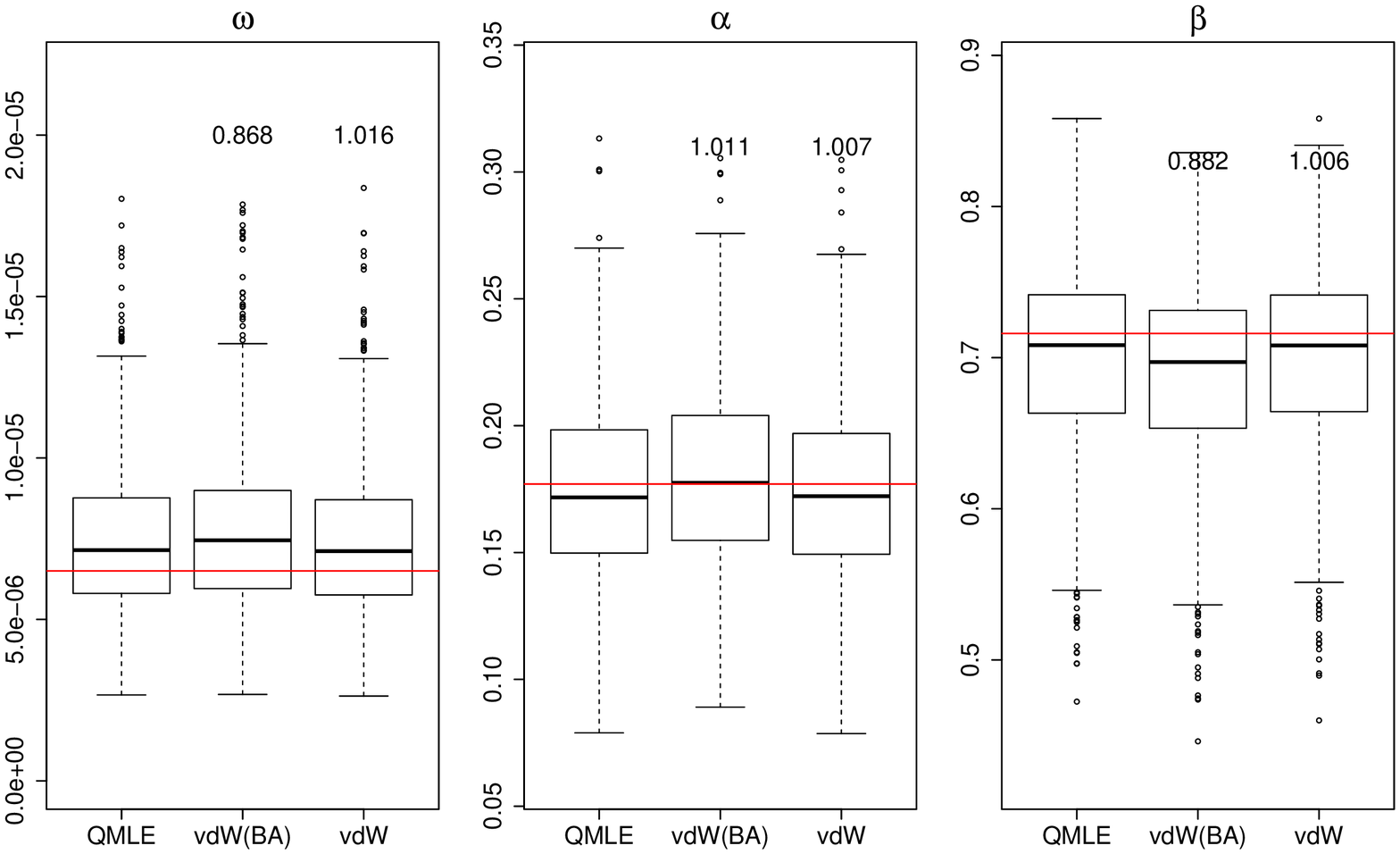}\vspace{-1mm}
\end{tabular}
\end{center}
\vspace{-3mm}\label{Fig: BA}
\end{figure}

To illustrate the advantages of our proposed R-estimator over the commonly-used QMLE and R-estimator of Andrews (2012) (BA, henceforth), we consider below some simulation results corresponding to the GARCH~($1, 1$) model with underlying standardized innovation density being (i) the standard normal distribution and (ii) the skew normal distribution (see e.g. Azzalini and Dalla Valle (1996) for details of such distribution). We generate $R=1000$ samples of size $n=1000$ with parameter values $(\omega, \alpha, \beta)^\prime = (6.50\times 10^{-6}, 0.177, 0.716)^\prime$ as in Section \ref{sim.Rest}. Simulation results described below are similar to various other choices of the true parameters. To make a fair comparison, for both R-estimators we use the van der Waerden score (vdW, henceforth); our vdW score is given in Section \ref{sec.exam} while BA's one has a different form; see Andrews (2012, Section 3) and Section \ref{subsec.cvp} for how to obtain a consistent estimator of the unknown parameters. Note that, only R-estimator of $(\alpha/\omega, \beta)^\prime$ is given in Andrews (2012); to compare each component, we provide a consistent estimator of $\omega$  for her R-estimator in Section \ref{subsec.cvp}. The resulting boxplots of all estimators are displayed in Figure \ref{Fig: BA} under the skew normal (upper panel) and standard normal (lower panel) innovation densities and the MSE ratios of the QMLE over other estimators are reported.

An inspection of these plots reveals superiority of our R-estimator over the QMLE and BA's. Under the normal error distribution, the distribution patterns of the R-estimators are quite similar to the QMLE around the true parameter value, and the MSE ratios of the QMLE over the R-estimators are all close to one. However, under the skew normal errors where asymmetry is introduced, our R-estimator has the least dispersion and with a gain of around $50\%$ efficiency over the QMLE, while the BA's R-estimator still has similar efficiency as the QMLE for $\omega$ and $\beta$. Therefore,  although the BA's R-estimator achieves efficiency as ours under the normal distribution, it is not as efficient as ours under an asymmetric distribution.

\subsection{Outline of the paper}

The rest of the paper is organized as follows. Section \ref{sec.defR.garch} defines a class of scale-transformed R-estimators based on an asymptotic linearity result of a rank-based central sequence. A consistent estimator of the unknown scalar is provided, and the asymptotic distributions and efficiency of the resulting R-estimators are discussed.  Also, we give an algorithm for computing the R-estimators. Section \ref{sec.num.R.GARCH} contains empirical and simulation results of the R-estimators. Section \ref{sec.Boot} describes the weighted bootstrap for the R-estimators and includes extensive simulation results. Section \ref{sec.GJR.R.main} considers an application to the GJR model. Conclusion is given in Section \ref{sec.con.R.GARCH}. The technique used to establish the asymptotic distribution is included in Appendix~\ref{sec.proof}.

\section{The class of R-estimators for the GARCH model}\label{sec.defR.garch}
In this section, we first define a central sequence of R-criteria $\{\h{\mbf R}_n(\bth)\}$ based on ranks of the residuals of the GARCH model. We prove the asymptotic uniform linear expansion (\ref{Prop.asy.lin1}) of this central sequence, which enables us to define one-step R-estimator $\h{\bth}_{n\vp}$ in (\ref{onestep.def}) as a root-$n$ consistent estimator of a scale-transformed GARCH parameter
\begin{equation}\label{bth0vp}
\bth_{0\vp}=\left( c_\vp\omega_0, c_\vp\alpha_{01}, ..., c_\vp\alpha_{0p}, \beta_{01}, ..., \beta_{0q} \right)^{'}
\end{equation}
with a constant $c_\vp>0$ satisfying (\ref{eq:IdenCon}). Based on $\h{\bth}_{n\vp}$ and $\{X_t\}$, we are able to derive a consistent estimator of $c_\vp$ and thus a root-$n$ consistent R-estimator $\hat{\bth}_n$ of $\bth_0$. We discuss some computational aspects and propose a recursive algorithm for computation in Section \ref{com.asp}.

Notations: Throughout the paper, for a function $g$, we use $\dot{g}$ and $\ddot{g}$ to denote its first and second derivatives whenever they exist. We use 
$c$, $b$, $c_1$ to denote positive constants whose values can possibly change from line to line. Let $\epsilon$ be a generic random variable (r.v.) with the same distribution as $\{\epsilon_t\}$ and let $F$ and $f$ denote the cumulative distribution function (c.d.f.) and probability density function (p.d.f.) of $\epsilon$, respectively. Let $\eta_t := \epsilon_t/\sqrt{c_\vp}$ and $\eta$ be a generic r.v. with the same distribution as $\{\eta_t\}$. Let $G$ and $g$ be the c.d.f. and p.d.f. of $\eta$, respectively. A sequence of stochastic process $\{Y_n(\cdot)\}$ is said to be $u_{\rm P}(1)$ (denoted by $Y_n=u_{\rm P}(1)$) if for every $c>0$, 
$\sup \{|Y_n(\bb)|; ||\bb|| \le c\}=o_{\rm P}(1)$, where $||\cdot||$ stands for the Euclidean norm.

\subsection{Rank-based central sequence}

From Lemma 2.3 and Theorem 2.1 of Berkes et al.~(2003),
$\sigma^2_t$ of (\ref{m2}) has the unique almost sure representation
$\sigma^2_t=c_0+\sum_{j=1}^{\infty} c_j X^2_{t-j},\,\,\, t \in \mathbb{Z}$,
where $\{c_j; j \ge 0\}$ are defined in (2.7)-(2.9) of Berkes et al.~(2003). 

Let $\bth_0=(\omega_0, \alpha_{01}, \ldots, \alpha_{0p}, \beta_{01}, \ldots, \beta_{0q})^{\prime}$ denote the true parameter belonging to a compact subset 
$\bTheta$ of $(0, \infty)^{1+p} \times (0, 1)^{q}$. A typical element in $\bTheta$ is denoted by 
$$\bth=(\omega, \alpha_{1}, \ldots, \alpha_{p}, \beta_{1}, \ldots, \beta_{q})^{\prime}.$$ Define the variance function by
$$
v_t(\bth)=c_0(\bth)+\sum_{j=1}^{\infty} c_j(\bth) X_{t-j}^2, \,\,\,\,\, 
\bth \in \bTheta, t \in \mathbb{Z}, 
$$
where the coefficients $\{c_j (\bth); j \ge 0\}$ are given in (3.1) of Berkes et al.~(2003) with the property 
$c_j (\bth_0)=c_j$, $j \ge 0$, 
so that the variance functions satisfy $v_{t}(\bth_0)=\sigma_t^2$, $t \in \mathbb{Z}$ and
\begin{equation*}
X_t= \{v_{t}(\bth_0)\}^{1/2}\epsilon_t, \;\; 1 \le t \le n.
\end{equation*}
Let $\{\h{v}_t(\bth)\}$ be observable approximation of $\{v_t(\bth)\}$, which is defined by 
$$
\h{v}_t(\bth)=c_0(\bth)+I(2 \le t) \sum_{j=1}^{t-1} c_j (\bth) X^2_{t-j}, \,\, \bth \in \bTheta, \,\, 1 \le t \le n.
$$
Let $H^*(x) = x\{-\d{f}(x)/f(x)\}$. The maximum likelihood estimator (MLE) is a solution of $\bDelta_{n, f}(\bth) = \0$, where 
\begin{equation*}
\bDelta_{n, f}(\bth) := n^{-1/2} \s \frac{\d{\mbf v}_t(\bth)}{v_t(\bth)}\left\lbrace 1-H^*\left[ \frac{X_t}{v_t^{1/2}(\bth)} \right] \right\rbrace.
\end{equation*}
However, $f$ in $H^*$ is usually unknown and we therefore consider an approximation to 
$\bDelta_{n, f}(\bth)$. 

Let $\vp: (0, 1) \rightarrow \R$ be a score function satisfying some regularity conditions which will be discussed later. Examples of $\vp$ are given in Section \ref{sec.exam}. Let $R_{nt}(\bth)$ denote the rank of $X_t/v_t^{1/2}(\bth)$ among $\{X_j/v_j^{1/2}(\bth); 1\leq j \leq n\}$. In linear regression models, the MLE has the same asymptotic efficiency as an R-estimator based on the score function $\vp(u) = -\d{f}(F^{-1}(u))/f(F^{-1}(u)) $. For the estimation of the scale parameters, the MLE corresponds to the {\it central sequence} 
\begin{equation}\label{R}
{\mbf R}_n(\bth):={\mbf R}_{n, \vp} (\bth) = n^{-1/2}\s \frac{\d{\mbf v}_t(\bth)}{v_t(\bth)} \left\lbrace 1-\vp\left[\frac{R_{nt}(\bth)}{n+1}\right]\frac{X_t}{v_t^{1/2}(\bth)}\right\rbrace.
\end{equation}
However, since $v_t(\bth)$ is unobservable, we therefore replace it by $\h{v}_t(\bth)$. 
Let $\h{R}_{nt}(\bth)$ denote the rank of $X_t/\h{v}_t^{1/2}(\bth)$ among $\{X_j/\h{v}_j^{1/2}(\bth); 1\leq j \leq n\}$.
We define {\it rank-based central sequence} as
\begin{equation}\label{def.R}
\h{\mbf R}_{n}(\bth):= \h{\mbf R}_{n, \vp}(\bth)=n^{-1/2}\s \frac{\d{\h{\mbf v}}_t(\bth)}{\h{v}_t(\bth)}\left\lbrace 1-\vp\left[\frac{\h{R}_{nt}(\bth)}{n+1}\right]\frac{X_t}{\h{v}_t^{1/2}(\bth)}\right\rbrace.
\end{equation}

\subsection{One-step R-estimators and their asymptotic distributions}

As we will show in this subsection, our R-estimator $\hat{\bth}_n$ of $\bth_0$ is obtained by estimating first  $\bth_{0\vp}$ and then the unknown scalar $c_{\vp}$.

\subsubsection{One-step R-estimators of $\bth_{0\vp}$}

To define the R-estimator in terms of the classical Le Cam's one-step approach as in Hallin and La Vecchia~(2017) and Hallin et al.~(2019), we derive the asymptotic linearity of the rank-based central sequence under the following assumptions. Let $c_\vp>0$ be defined by
$$
\sqrt{c_\vp} = {\mathrm E}\left[ \vp\left(F \left(\epsilon_t \right)\right)\epsilon_t \right] 
$$
so that
\begin{equation}\label{eq:IdenCon}
{\mathrm E}\left\lbrace\vp\left[F \left(\epsilon_t \right)\right] \frac{\epsilon_t}{\sqrt{c_\vp}}\right\rbrace=1. 
\end{equation}
Define $\mu(x) := \int_{-\infty}^x s g(s) ds$. Since $g(x)>0$, $\mu(x)$ is strictly decreasing on $(-\infty, 0]$ with range 
$[\mu(0), 0]$ and strictly increasing on $[0, +\infty)$ with range $[\mu(0), 0]$. The functions $y \rightarrow \mu^{-1}(y)$ on
$[\mu(0), 0]$ with ranges $(-\infty, 0]$ and $[0, +\infty)$ are well-defined when the ranges are considered separately. 

The following conditions on the distribution of $\eta_t$ are assumed for the proof of Theorem~
\ref{Thm.emp} in Appendix on the approximation of a scale-perturbed weighted mixed-empirical process by its non-perturbed version. 

\noindent\textbf{Assumption (A1)}. 
(i). The function $x^2 g(x)$ is bounded on $x \in \mathbb{R}$ (and consequently so are the functions $g(x)$ and $xg(x)$); functions $y \rightarrow \mu^{-1}(y) g(\mu^{-1}(y))$ and 
$y \rightarrow (\mu^{-1}(y))^2 g(\mu^{-1}(y))$ are uniformly continuous on $[\mu(0), 0]$ when the ranges are considered separately as in the definition of $\mu$ above;

(ii). 
$$\underset{\delta \rightarrow 0}{\lim} \sup \left\lbrace |x| \int_0^1 |xg(x) - (x + hx\delta) g(x + hx\delta)| dh; x \in \mathbb{R} \right\rbrace = 0;
$$ 
(iii). There is a $\delta >0$ such that ${\mathrm E}|\eta_t|^{2+\delta} < \infty$.\medskip

We remark that Assumption (i) entails that $\mu(x)$ is uniformly Lipschitz continuous in scale in the sense that for some constant 
$0< c < \infty$ and for every $s \in \mathbb{R}$, we have $\sup_{x \in \mathbb{R}} |\mu (x + xs) - \mu (x)| \leq c|s|$.

A more easily verifiable condition for Assumption (ii) can be obtained, for example, when $g$ admits the derivative $\dot{g}$ which satisfies that for some $\delta > 0$,
$$\sup\{ x^2 \sup|g(y) + y\dot{g}(y)|; x(1-\delta) < y < x(1+\delta) \} < \infty.$$
In particular, Assumptions (i), (ii) and (iii) in (A1) hold for a wide range of error distributions, including normal, double-exponential, logistic and $t$-distributions with degrees of freedom more than $2$ which are considered for simulation study.

We also need the following assumptions on the parameter space and the score function $\vp$.

\noindent\textbf{Assumption (A2)}. Denoting by $\bTH_0$ the set of interior points of $\bTH$, we assume that $\bth_{0}, \bth_{0\vp} \in \bTH_0.$\medskip

\noindent\textbf{Assumption (A3)}. The score function $\vp$ is non-decreasing, right-continuous with only a finite number of points of discontinuity and is bounded on $(0, 1). $ \medskip

We now compare our assumptions with those made by Andrews (2012), to be called as BA1, BA2 etc.
Assumption BA1 states the stationarity and ergodicity of $\{X_t\}$ similar to what we assume but not necessarily finiteness of the variance of $X_t$. However, estimation of the intercept parameter $\alpha_{00}$ appearing in Andrews (2012, Equation (2.2)) is not considered there. In this paper, we estimate the intercept parameter under the finite variance assumption of 
$X_t$ as in (A4). Higher moment assumptions were also made in Fan et al.~(2014) while estimating the equivalent scale parameter $\eta_f$. 

Assumptions BA2 and BA3 are related to the uniqueness of parameters and the existence of the non-degenerate errors and observations. We assume non-degenerate models as is common in the literature. Assumptions BA4 and BA7 are on the score functions which are bounded, non-decreasing and left-continuous. In (A3), we assume that the score functions are bounded, non-decreasing and right-continuous with finite number of discontinuity. Assumptions BA5 and BA6 are on the cdf and pdf of the log of the squared error distribution. We have made analogous assumptions on the error pdf itself and uniform continuity on the $\mu^{-1}$-transformed axis in (A1)(i) and (A1)(ii); the later came up as a part of the technical assumptions in using some convergence results of empirical processes to derive the asymptotic distribution of the R-estimators. Assumptions BA8, BA9 and BA10 describe various scenarios related to some of the component parameters equal to zero as in Francq and Zakoian (2007). In (A2), we assume that all parameters are positive as in Berkes and Horváth (2004), Mukherjee (2008) and Francq et al.~(2011) in relation to the M-score and this corresponds to BA8. 

To state the asymptotic linearity of $\h{\mbf R}_n(\bth)$, we introduce the following notation. Let
$$
\gamma(\vp) := \int_0^1 \int_0^1 G^{-1}(u) G^{-1}(v) \left[  \min\{u, v\} - uv \right] d \vp(u) d \vp(v),
\,\, \J(\bth) :={\mathrm E}(\d{\mbf v}_t (\bth) \d{\mbf v}_t^{'}(\bth)/v_t^2 (\bth)) 
$$
$$\rho(\vp) := \int_0^1 \{G^{-1}(u)\}^2 g\{G^{-1}(u) \} d \vp(u),\,\,
\sigma^2(\vp) := {\mathrm E}\left\lbrace \vp\left[G(\eta_t) \right] \eta_t \right\rbrace^2 - 1,
$$
\begin{equation}\label{lam.def}
\lambda(\vp):= \int_{0}^{1}  \int_{0}^{1} G^{-1}(u) I(v \leq u) (1- G^{-1}(v) \vp(v)) dv d\vp(u).
\end{equation}
Let $Z$ be the r.v. $Z:= \int_0^1 G^{-1}(u) B(u) d\vp(u)$,
where $B(.)$ is the standard Brownian bridge. Then $Z$ has mean zero and variance $\gamma(\vp)$; see the proof in  Lemma~
\ref{lemmaRT.diff} for details. Let $\tilde{G}_n(x)$, $x\in \R$ be the empirical distribution function of $\{\eta_t\}$ (which is unobservable), 
$$
\Q_n(\bth) := \int_{0}^{1}  n^{-1/2} \s \frac{\d{\mbf v}_t (\bth)}{v_t (\bth)} \left[ \mu(G^{-1}(u)) -  \mu(\tilde{G}_n^{-1}(u))\right]d \varphi(u), 
$$
$$
\N_n (\bth) := n^{-1/2} \s \frac{\d{\mbf v}_t (\bth)}{v_t (\bth)}\left\lbrace 1-\eta_t \vp\left[ G(\eta_t) \right] \right\rbrace.
$$
The following proposition states the asymptotic uniform linearity of $\h{\mbf R}_n(\bth)$.

\begin{proposition}\label{Prop.asy.linear}
Let Assumptions (A1)-(A3) hold. Then for $\bb \in \mathbb{R}^{1+p+q}$ with $||\bb|| < c$,
\begin{equation}\label{Prop.asy.lin1}
\h{\mbf R}_n(\bth_{0\vp} + n^{-1/2} \bb) - \h{\mbf R}_n(\bth_{0\vp}) = (1/2 + \rho(\vp)/2)\J (\bth_{0\vp}) \bb + u_{\rm P}(1).
\end{equation}
Moreover,
\begin{equation}\label{Prop.asy.lin2}
\h{\mbf R}_n(\bth_{0\vp}) = \Q_n(\bth_{0\vp}) + \N_n(\bth_{0\vp}) + u_{\rm P}(1),
\end{equation}
where $\Q_n(\bth_{0\vp}) $ converges in distribution to ${\mathrm E}(\dot{\mbf v}_1(\bth_{0\vp}) /v_1(\bth_{0\vp})) Z$ with mean zero and covariance matrix ${\mathrm E}(\dot{\mbf v}_1(\bth_{0\vp}) /v_1(\bth_{0\vp}) ) {\mathrm E}(\dot{\mbf v}_1^\prime (\bth_{0\vp}) /v_1(\bth_{0\vp}) ) \gamma(\vp)$ and $\N_n(\bth_{0\vp})  \rightarrow {\mathcal N}( \0, \J(\bth_{0\vp})  \sigma^2(\vp))$.
\end{proposition}
The above asymptotic linearity allows us to define a class of R-estimators through the one-step approach. Let $\{\h{\bUpsilon}_n\}$ be a sequence of consistent estimator of 
$\bUpsilon_{\vp, g}(\bth_{0\vp}):= (1/2 + \rho(\vp)/2)\J(\bth_{0\vp}) $; see Section \ref{com.asp} for a construction of $\h{\bUpsilon}_n$. Let $ \bar{\bth}_n$ be a root-$n$ consistent estimator of $\bth_{0\vp}$ and, for technical reasons, we assume $ \bar{\bth}_n$ is asymptotically discrete. More precisely,  a sequence $\{\bar{\bth}_n\}$ is called discrete if there exists $K \in \mathbb{N}$ such that independent of $n \in \mathbb{N}$, $\bar{\bth}_n$ takes on at most $K$ different values in 
$$\mathcal{Q}_n := \{\bth \in \mathbb{R}^{1+p+q}: n^{-1/2} \left\Vert \bth - \bth_0 \right\Vert \leq c\}, \quad c > 0 \, \, \text{fixed};$$
see Kresis (1987, Section 4) for details. We remark that here asymptotically discreteness is only of theoretical interest since in practice 
$\bar{\bth}_n$ always has a bounded number of digits; see Le Cam and Yang~(2000, Chapter~6) and van der Vaart (1998, Section~5.7)  for more details. Then the one-step R-estimator is defined as 
\begin{equation}\label{onestep.def}
\h{\bth}_{n\vp} :=  \bar{\bth}_n - n^{-1/2} \left( \h{\bUpsilon}_n \right)^{-1} \h{\mbf R}_n ( \bar{\bth}_n).
\end{equation}

Note that strictly speaking, the R-estimators based on this definition are not functions of the ranks of the residuals only. However, we borrow the terminology from the regression and the homoscedastic-autoregression settings and still call them (generalized) R-estimators.  When, for example, $\varphi(u)= u-1/2$, $\h{\bth}_{n\vp}$ is an analogue of the Wilcoxon type R-estimator. 

The following theorem shows that the R-estimator defined in (\ref{def.R}) is $\sqrt{n}$-consistent estimator of $\bth_{0\vp}$. The proof is given in Appendix~
\ref{sec.proof}.

\begin{theorem}\label{thm.asy}
Let Assumptions (A1)-(A3) hold. Then, as $n \rightarrow \infty$,
\begin{equation}\label{asy}
\sqrt{n} \left(\h{\bth}_{n\vp} - \bth_{0\vp} \right) = -(1/2 + \rho(\vp)/2)^{-1} (\J(\bth_{0\vp}))^{-1}  (\Q_n(\bth_{0\vp})  + \N_n(\bth_{0\vp}) ) + o_{\rm P}(1).
\end{equation}
Hence as $n \rightarrow \infty$, $\sqrt{n} \left(\h{\bth}_{n\vp} - \bth_{0\vp} \right)$ is normal with mean $\0$ and covariance matrix
$$(\J(\bth_{0\vp}))^{-1} \frac{\left[ 4\gamma(\vp)  + 8  \lambda(\vp) \right]  {\mathrm E}(\d{\mbf v}_1(\bth_{0\vp})/v_1(\bth_{0\vp})) {\mathrm E}\left( \d{\mbf v}_1^\prime(\bth_{0\vp}) /v_1(\bth_{0\vp}) \right) + 4 \sigma^2(\vp) \J(\bth_{0\vp}) }{(1 + \rho(\vp))^2}  (\J(\bth_{0\vp}))^{-1}.$$
\end{theorem}

\subsection{Estimation of $c_{\vp}$}\label{subsec.cvp}

To obtain a root-$n$ consistent estimator of $\bth_0$, we estimate the unknown scalar $c_{\vp}$ under the following mild moment assumption on $X_t$.

\noindent\textbf{Assumption (A4)}. ${\rm E}(X_t^2) < \infty.$\medskip

From Bollerslev (1986), a sufficient condition for the GARCH model to be  second-order stationary (hence to satisfy Assumption (A4)) is $\sum_{i = 1}^p \alpha_{0i} + \sum_{j=1}^q \beta_{0j} < 1.$ In this case,
$$
{\rm E}(X_t^2)=\frac{\omega_0}{1- \sum_{i = 1}^p \alpha_{0i} - \sum_{j=1}^q \beta_{0j}} 
= \frac{c_{\vp}\omega_0}{c_{\vp} - \sum_{i = 1}^p (c_{\vp} \alpha_{0i}) - c_{\vp}\sum_{j=1}^q \beta_{0j}}. 
$$
Using the ergodicity property, ${\rm E}(X_t^2)$ can be estimated from the data by $\overline{X_n^2} := n^{-1} \sum_{t=1}^n X_t^2$.  Also, 
$c_{\vp}\omega_0$ and 
$c_{\vp}\alpha_{0i}$ can be estimated by $\hat{\omega}_{n\vp}$ and $\hat{\alpha}_{n\vp i}$, respectively. Solving the following equation for $c$
$$
\sum_{t=1}^n X_t^2/n = \frac{\hat{\omega}_{n\vp}}{c - \sum_{i = 1}^p \hat{\alpha}_{n\vp i} - c \sum_{j=1}^q \hat{\beta}_{n j}},
$$
we obtain an estimate of $c_{\vp}$ as 
\begin{equation}\label{hatcvp}
\hat{c}_{n\vp} :=  \left(1 -  \sum_{j=1}^q  \hat{\beta}_{nj} \right)^{-1} \left( \hat{\omega}_{n\vp}/\overline{X_n^2}  + \sum_{i = 1}^p  \hat{\alpha}_{n\vp i} \right). 
\end{equation}   
Consequently write $\hat{\bth}_{n\vp}$ in its component-wise form
$$\hat{\bth}_{n\vp} = \left( \hat{\omega}_{n\vp},  \hat{\alpha}_{n\vp 1}, ..., \hat{\alpha}_{n\vp p}, \hat{\beta}_{n1}, ..., \hat{\beta}_{nq} \right)^{'}$$ and let
$$\hat{\bth}_n := \left( \hat{\omega}_{n\vp}/\hat{c}_{n\vp},  \hat{\alpha}_{n\vp 1}/\hat{c}_{n\vp}, ..., \hat{\alpha}_{n\vp p}/\hat{c}_{n\vp}, \hat{\beta}_{n1}, ..., \hat{\beta}_{nq} \right)^{'}.$$
The following theorem states that $\hat{c}_{n\vp}$ is a consistent estimator of $c_{\vp}$ and thus $\hat{\bth}_n$ is a root-$n$ consistent estimator of 
$\bth_0$ with asymptotically normal distribution; see Appendix~\ref{sec.proof} for its proof.

\begin{theorem}\label{thm.asybth0}
Let Assumptions (A1)-(A3) hold. Then, as $n \rightarrow \infty$,
$$\hat{c}_{n\vp} = c_{\vp} + o_{\rm P}(1)$$
and
\begin{equation}\label{asybth0}
\sqrt{n} \left(\h{\bth}_{n} - \bth_{0} \right) = -(1/2 + \rho(\vp)/2)^{-1} (\J(\bth_{0}))^{-1}  (\Q_n(\bth_{0})  + \N_n(\bth_{0}) ) + o_{\rm P}(1).
\end{equation}
Hence as $n \rightarrow \infty$, $\sqrt{n} \left(\h{\bth}_{n} - \bth_{0} \right)$ is normal with mean $\0$ and covariance matrix
$${\boldsymbol \Omega}_{\vp} := (\J(\bth_{0}))^{-1} \frac{\left[ 4\gamma(\vp)  + 8  \lambda(\vp) \right]  {\mathrm E}(\d{\mbf v}_1(\bth_{0})/v_1(\bth_{0})) {\mathrm E}\left( \d{\mbf v}_1^\prime(\bth_{0}) /v_1(\bth_{0}) \right) + 4 \sigma^2(\vp) \J(\bth_{0}) }{(1 + \rho(\vp))^2}  (\J(\bth_{0}))^{-1}.$$

\end{theorem}

Note that by assuming (A4) and following similar approach as in the proof of Theorem \ref{thm.asybth0}, one can also obtain consistent estimators of unknown scalars for robust estimators in GARCH models (e.g., M-estimators of Liu and Mukherjee (2020) and R-estimator of Andrews (2012)). For illustration, recall that the R-estimator of Andrews (2012) is root-$n$ consistent estimator of 
$$(\alpha_{01}/\omega_0, ..., \alpha_{0p}/\omega_0, \beta_{01}, ..., \beta_{0q})^\prime$$
where $\omega_0$ is the intercept parameter that plays the role of the unknown scalar. Denoting the R-estimator of Andrews (2012) by 
$$
\hat{\bth}_{\text{BA};n} := (\hat{a}_{\text{BA}; n1}, ..., \hat{a}_{\text{BA}; np}, \hat{\beta}_{\text{BA}; n1}, ..., \hat{\beta}_{\text{BA}; nq})^\prime, $$
a consistent estimator of $\omega_0$ can be obtained following similar approach as in the proof of Theorem \ref{thm.asybth0} as 
$$\hat{\omega}_{\text{BA};n} := \left( 1 + \sum_{i=1}^p \hat{a}_{\text{BA}; ni}  \overline{X_n^2} \right)^{-1}  \left(1 -  \sum_{j=1}^q  \hat{\beta}_{\text{BA}; nj} \right) \overline{X_n^2}.$$

We remark that from Theorem \ref{thm.asybth0}, the asymptotic covariance matrix of $\h{\bth}_{n}$ has a complicated form. Hence we consider bootstrap methods in Section \ref{sec.Boot} to approximate the limit distribution of $\sqrt{n} \left(\h{\bth}_{n} - \bth_{0} \right)$.


\subsection{Examples of the score functions}\label{sec.exam}

Below we cite examples of three commonly-used R-scores; for similar examples of scores in other models, see Mukherjee~(2007) and Hallin and La Vecchia~(2017).

\textbf{Example 1} (sign score). Let $\vp(u)= \text{sign}(u-1/2)$. Then for symmetric innovation distribution, $c_\vp=({\mathrm E}|\epsilon|)^2,$ which coincides with the scale factor of the LAD estimator in Mukherjee~(2008). Therefore, the sign R-estimator is expected to be close to the LAD estimator. This is demonstrated later in the real data analysis.

\textbf{Example 2} (Wilcoxon score). Let $\vp(u)= u-1/2$ so that the range of $\vp(u)$ is symmetric. 

\textbf{Example 3} (van der Waerden (vdW) or normal score). One might also set $\vp(u) = \Phi^{-1}(u),$ with $\Phi(\cdot)$ denoting the c.d.f. of the standard normal distribution. Notice that unlike the sign and Wilcoxon score, the vdW score is not bounded as $u \rightarrow 0$ and $u \rightarrow 1$. It thus does not satisfy Assumption (A3). However, an approximating sequence of bounded score functions of $\vp$ on $(0, 1)$ can be constructed as in Andrews~(2012). For example, letting 
$$\vp_{m}(u) := \Phi^{-1}(u) I(1/m \leq u \leq 1 - 1/m) + \Phi^{-1}(1/m) I(0<u< 1/m) + \Phi^{-1}(1-1/m) I(u>1-1/m)$$ with $m>2$, then $\vp_{m}(u)$ satisfies Assumption (A3) and converges pointwise to the vdW score on $(0, 1)$.
It is demonstrated later using both real data analysis and extensive simulation that the vdW has superior performance compared with the QMLE.

We now provide heuristics for the definition of the R-estimator in (\ref{R}). 
When the underlying error distribution is known, one can obtain efficient R-estimator by choosing the score function as 
$\vp(u) = -\d{f}(F^{-1}(u))/f(F^{-1}(u))$. 
Since for large $n$, the empirical distribution function $R_{nt}(\bth_{0\vp})/(n+1)$ of $\{\epsilon_j; 1 \le j \le n\}$ evaluated at $\epsilon_t$ is close to $F(\epsilon_t)$, we have 
$$
\vp\left[\frac{R_{nt}(\bth)}{n+1}\right]\frac{X_t}{v_t^{1/2}(\bth)}
\approx H^*\left[ \frac{X_t}{v_t^{1/2}(\bth)} \right].
$$
Therefore, the criteria function of the R-estimator gets close to the MLE which is efficient. This leads to the choice of the vdW, sign and Wilcoxon under the normal, double exponential (DE) and logistic distributions, respectively. This is observed later in simulation study of the R-estimator. 

\subsection{Computational aspects}\label{com.asp}
Here we discuss some key computational aspects and propose an algorithm to compute $\h{\bth}_{n\vp}$ and $\h{\bth}_{n}$.

First, since $c_{\vp}$ depends on the unknown density $f$, it is difficult to have a $\sqrt{n}$-consistent initial estimator  $\bar{\bth}_n$ of $\bth_{0\vp}$. However, due to finite sample size in practice, the one-step procedure is usually iterated a number of times, taking $\h{\bth}_{n\vp}$ as the new initial estimate, until it stabilizes numerically. This iteration process would mitigate the impact of different initial estimates; see van der Vaart (1998, Section~5.7) and Hallin and La Vecchia~(2017) for similar comments. In fact, we observed during our extensive simulation study that irrespective of the choice of the QMLE, LAD or $\bth_0$ as initial estimates, only few iterations result in the same estimates.

Second, to compute $\h{\bth}_{n\vp}$ of (\ref{onestep.def}), we need $\h{\bUpsilon}_n$ which is a consistent estimator of $(1/2 + \rho(\vp)/2)\J(\bth_{0\vp})$. The matrix 
$\J(\bth_{0\vp})$ can be consistently estimated by 
$$\h{\J}_n(\bar{\bth}_n) := n^{-1} \sum_{t=1}^n \{\dot{\h{\mbf v}}_t( \bar{\bth}_n) \dot{\h{\mbf v}}_t^\prime ( \bar{\bth}_n)/\h{v}_t^2 ( \bar{\bth}_n)\}.$$
For estimating $\rho(\vp)$ which is a function of the density $g$, we can use the asymptotic linearity in (\ref{Prop.asy.lin1}). Here with an arbitrarily chosen $\bb$, we can substitute $ \bar{\bth}_n$ for $\bth_{0\vp}$ and then solve the equation for $\rho(\vp)$ based on (\ref{Prop.asy.lin1}). A more delicate approach for estimating $\rho(\vp)$ can be found in  Cassart et al.~(2010) and Hallin and La Vecchia~(2017, Appendix~C). Based on our  extensive simulation study and real data analysis, it appears that different values of $\rho(\vp)$ would finally lead to same estimate after some iterations. 
Consequently, we set $\rho(\vp)=1$ during the computation which is the value corresponding to the vdW score under the normal distribution.

In summary, we propose the following iterative Algorithm~\ref{algRestGARCH} to compute $\h{\bth}_{n\vp}$, with which we can obtain $\hat{c}_{\vp}$ using \eqref{hatcvp} and hence $\h{\bth}_n$. Codes are available upon request.


\begin{algorithm}[!htbp]
\SetAlgoLined\vspace{3mm}

\KwIn{a sample $\{X_t; 1 \leq t \leq n\}$, orders $p$ and $q$ of the GARCH process, number $k$ of iterations in the one-step procedure.}
\KwOut{R-estimator $\hat{\bth}_n$}
\begin{enumerate}

\item Compute a preliminary root-$n$ consistent estimator $\bar{\bth}_n$ and set $\hat{\bth}_{n\vp} = \bar{\bth}_n$.
 
\item \For{$i \gets 1$ \textbf{to} $k$}{

\begin{equation}
\begin{split}
\h{\bth}_{n\vp} &\gets \h{\bth}_{n\vp}-\left[\sum_{t=1}^n \frac{\d{\h{\mbf v}}_t(\h{\bth}_{n\vp})\d{\h{\mbf v}}_t^{'}(\h{\bth}_{n\vp})}{\h{v}_t^2(\tilde{\bth}_{n\vp})}\right]^{-1} \\
 & \quad \quad \quad \quad \times\left\lbrace\sum_{t=1}^n \frac{\d{\h{\mbf v}}_t(\h{\bth}_{n\vp})}{\h{v}_t(\tilde{\bth}_{n\vp})}\left[1-\varphi\left(\frac{R_{nt}(\h{\bth}_{n\vp})}{n+1}\right)\frac{X_t}{\h{v}_t^{1/2}(\h{\bth}_{n\vp})}\right]\right\rbrace
\end{split} 
\label{eq:formula}
\end{equation}
}

\item Compute  $\hat{c}_{\vp}$ using \eqref{hatcvp} and then $\h{\bth}_n$.

\end{enumerate}

\caption{R-estimation for GARCH models}\label{algRestGARCH}
\end{algorithm}

\subsection{Asymptotic relative efficiency}\label{sec.ARE}
In the linear regression and autoregressive models, the asymptotic relative efficiency (ARE) of the R-estimators with respect to (wrt) the least squares estimator is high for a wide array of error distributions. For the GARCH model, we compare the ARE of the R-estimator wrt the QMLE based on Theorem~\ref{thm.asybth0}. 

%


Note that under assumption ${\mathrm E}\epsilon^4 < \infty$, the QMLE $\h{\bth}_{\text{QMLE}}$ is asymptotic normal with mean zero and covariance matrix 
${\boldsymbol \Omega}_{\text{QMLE}} = ({\mathrm E}\epsilon^4 - 1) (\J(\bth_0))^{-1}.$
Hence, in view of Theorem~\ref{thm.asybth0},  the ARE of the R-estimator wrt the QMLE is 
\begin{align}
& {\boldsymbol \Omega}_{\vp}^{-1} {\boldsymbol \Omega}_{\text{QMLE}} =\J(\bth_0) \left\lbrace \left[ 4\gamma(\vp)  + 8  \lambda(\vp) \right]  {\mathrm E}\left( \frac{\d{\mbf v}_1(\bth_0)}{v_1(\bth_0)} \right) {\mathrm E}\left( \frac{\d{\mbf v}_1^\prime (\bth_0)}{v_1 (\bth_0)} \right)
 + 4 \sigma^2(\vp) \J(\bth_0) \right\rbrace^{-1} \nonumber  \\
 & \qquad \quad \qquad \times (1 + \rho(\vp))^2 ({\rm  E} \epsilon^4  - 1).  \label{ARE.cov}
\end{align}
For the sign R-estimator, $\gamma(\vp)$, $\lambda(\vp)$ and $\rho(\vp)$ are all zeros. Hence ${\boldsymbol \Omega}_{\vp}^{-1} {\boldsymbol \Omega}_{\text{QMLE}}$ reduces to 
$$({\mathrm E}\epsilon^4 - 1)/(4 \sigma^2(\vp)) \I_{1+p+q},$$
where $\I_{1+p+q}$ is the $(1+p+q)\times (1+p+q)$ identity matrix. Consequently, the ARE of the sign R-estimator wrt the QMLE 
equals $({\mathrm E}\epsilon^4 - 1)/(4 \sigma^2(\vp)),$
which is $0.876$ under the normal distribution. This corresponds to the classical result of the ARE of the mean absolute deviation wrt the mean square deviation; see, e.g., Huber and Ronchetti~(2011, Chapter~1).

For the vdW and Wilcoxon R-estimators, the AREs are more difficult to calculate since $\gamma(\vp)$ and $\lambda(\vp)$ are non-zero. However, in the following simulation study in Table~\ref{simulation1}, the estimated AREs reveal that the vdW R-estimator, compared with the QMLE, does not lose any efficiency which is a reflection of the well-known Chernoff-Savage phenomenon in the literature on the R-estimation in linear models.

\section{Real data analysis and simulation results}\label{sec.num.R.GARCH}
This section examines the performance of the R-estimators and compare them with the QMLE by analysing three financial time series and by carrying out extensive Monte Carlo simulation.

\subsection{Real data analysis}\label{sec.real}
In this section we fit GARCH~($1, 1$) model to three financial time series and compare the proposed three R-estimators with the 
M-estimators QMLE and LAD discussed in Mukherjee~(2008), where the unknown scalar of the LAD can also be estimated by \eqref{hatcvp}.

In an earlier work, Muler and Yohai~(2008) fitted the the GARCH~($1, 1$) model to the Electric Fuel Corporation (EFCX) time series for the period of January 2000 to December 2001 with sample size $n = 498$. The parameters of the model are estimated by 
M-estimators based on various score functions. It turned out that the M-estimates of the parameter $\beta$ differ widely depending on the score functions and so it is difficult to assess which estimate should be relied on in similar situations. Here we compare various M-estimates and R-estimates of the GARCH~($1, 1$) parameters for the EFCX series again shedding light on which could be some possible reasons for the difference in estimates and finally which estimation methods can be relied upon. We also compare M-estimates of the GARCH~($1, 1$) parameters when fitted to two other dataset, namely, the S\&P 500 stock index from June 2013 to May 2017 with $n = 1005$ and the GBP/USD exchange rate from June 2013 to May 2017 with $n = 998$ to illustrate that the M- and R-estimates do not differ widely when the underlying theoretical assumptions hold in general.  

In Table \ref{realdata}, we report the QMLE computed using the {\tt fGarch} package in R program, the M-estimates QMLE and LAD and the 
R-estimates proposed in Examples 1-3 of Section~\ref{sec.exam}. For the EFCX data, the R-estimates 
 for all score functions are quite close to the LAD estimate, but they are very different than the QMLE. On the contrary, for the S\&P 500 and GBP/USD data, M- and R-estimates  are close to each other. 
 
\begin{table}
\caption{The QMLE, LAD and R-estimates (sign, Wilcoxon and vdW) of the GARCH~(1, 1) parameter for the EFCX, S\&P 500 and GBP/USD data.} 
\label{realdata}
\begin{center}
\begin{tabular}{c c c c c c c }
\hline
  & {\tt fGarch} & QMLE & LAD & sign & Wilcoxon & vdW \\
\hline
\textbf{EFCX} & & &&&&\\
$\omega$ & 1.89$\times 10^{-4}$ & 6.28$\times 10^{-4}$ & 1.22$\times 10^{-3}$ & 1.23$\times 10^{-3}$ & 1.22$\times 10^{-3}$ & 1.24$\times 10^{-3}$ \\
$\alpha$ & 0.05     & 0.07     & 0.17     & 0.17     & 0.16     & 0.13     \\
$\beta$  & 0.92     & 0.84     & 0.66     & 0.65     & 0.67     & 0.69     \\ \hline
\textbf{S\&P 500} &          &          &          &          &          &          \\
$\omega$ & 6.50$\times 10^{-6}$ & 7.02$\times 10^{-6}$ & 5.31$\times 10^{-6}$ & 5.32$\times 10^{-6}$ & 5.32$\times 10^{-6}$ & 6.19$\times 10^{-6}$ \\
$\alpha$ & 0.18     & 0.18     & 0.19     & 0.19     & 0.19     & 0.18     \\
$\beta$  & 0.72     & 0.70     & 0.73     & 0.73     & 0.73     & 0.72     \\ \hline
\textbf{GBP/USD}  &          &          &          &          &          &          \\
$\omega$ & 5.32$\times 10^{-7}$ & 1.02$\times 10^{-6}$ & 6.83$\times 10^{-7}$ & 6.76$\times 10^{-7}$ & 7.29$\times 10^{-7}$ & 9.10$\times 10^{-7}$ \\
$\alpha$ & 0.12     & 0.13     & 0.07     & 0.07     & 0.07     & 0.09     \\
$\beta$  & 0.88     & 0.85     & 0.91     & 0.91     & 0.91     & 0.89 \\ \hline
\end{tabular}
\end{center}
\end{table}

To investigate why the QMLE  is different from the other R-estimates and LAD for the EFCX data, we check the assumption ${\mathrm E}\epsilon^4 < \infty$ for this data by using the QQ-plots of the residuals (based on the vdW
R-estimates) against $t$ distributions. We consider the vdW score only because the R-estimates based on two other score functions and the LAD are close to the vdW estimates. For comparison, we have also provided QQ-plots for the 
S\&P 500 data. The main idea behind the QQ-plots of the residuals against the $t(d)$ distribution is simple: Supposing
$\epsilon \sim t(d)$ distribution, then ${\mathrm E}|\epsilon|^\nu < \infty$ if and only if $\nu < d$. Hence, residuals with heavier tail than the $t(d)$ distribution correspond to the errors with the infinite $d$-th moment while those with thinner tail than the $t(d)$ distribution have the finite $d$-th error moment.

The top-left panel of Figure 1 shows the QQ-plot of the residuals against the $t(4.01)$ distribution for the EFCX data. The residuals have heavier right tail than the $t(4.01)$ distribution which implies that the fourth moment of the error term may not exist. On the other hand, the QQ-plot against the $t(3.01)$ distribution reveals lighter tail as shown at the bottom-left panel of Figure 1 and this implies that ${\mathrm E}|\epsilon|^3 < \infty$.
                  
For the S\&P 500 data, the QQ-plot against $t(4.01)$ distribution at the top-right panel of Figure 1 shows that the residuals have lighter tails than $t(4.01)$ distribution. For the QQ-plot against $t(6.01)$ distribution, as shown at the bottom-right panel of Figure 1, the residuals fit the distribution better. Therefore, we may conclude that ${\mathrm E}|\epsilon|^4 < \infty$ holds for the S\&P 500 data.

\begin{figure}[!htbp]
\begin{center}
\caption{QQ-plots of the residuals against $t$-distributions for the EFCX (left column) and S\&P 500 data (right column); the residuals are obtained by using the vdW R-estimator.}
\includegraphics[width=6in]{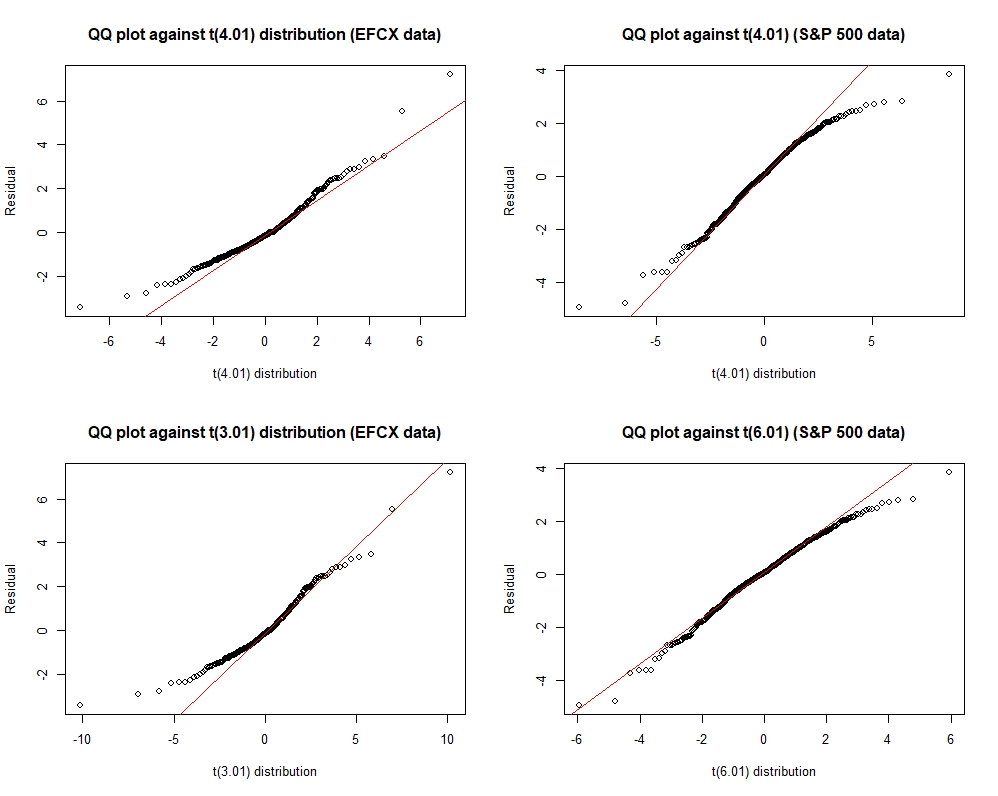}
\end{center}
\label{fig:QQplot}
\end{figure} \vspace{-2mm}

\subsection{Simulation study of the R-estimators}\label{sim.Rest}
We now evaluate the performance of the R-estimators based on simulated data from various error distributions. Apart from the GARCH~($1, 1$) model we consider the GARCH~($2, 1$) model also as the computation for higher order models are not considered frequently in the literature. Let $R$ denote the number of replications and
$\h{\bth}_{ni} = (\h{\omega}_i, \h{\alpha}_{i1}, ..., \h{\alpha}_{ip}, \h{\beta}_{i1}, ..., \h{\beta}_{iq})^\prime$ denote the 
R-estimator computed from the $i$-th data, $1 \le i \le n$. We throughout compare the R-estimators with the QMLE by using the averaged bias and MSE.
We also compare the relative efficiency of the R-estimators wrt the QMLE under a finite sample size, as an estimate of the ARE, by using the formula
$$\widehat{\text{ARE}}_{\text{R/QMLE}} = \widehat{\text{MSE}}_{\text{QMLE}}/ \widehat{\text{MSE}}_{\text{R}}.$$

\textbf{Simulation for the GARCH~($1, 1$) model.}
Here we run simulation with $R = 500, n = 1000$ and $\bth_0 = (6.50\times 10^{-6}, 0.177, 0.716)^\prime,$ where our choice of
$\bth_0$ is motivated by the estimate given by the {\tt fGarch} for the S\&P 500 data in Table 1. The estimates of the bias and MSE of the R-estimators and QMLE under various error distributions are reported in Table 
\ref{simulation1}, where the estimates of the ARE are shown in the parentheses. Notice that under  $t(3)$ distribution, the QMLE does not converge for many replications, while the R-estimators always converge. Therefore, the bias and MSE are obtained using the replications where the QMLE converges.

It is worth noting that the vdW achieves almost the same efficiency as the QMLE under the normal distribution, and the vdW is more efficient under heavier-tailed distributions.  In general, the sign score is most efficient under the DE and $t(3)$ distributions, while the Wilcoxon score is optimal under the logistic distribution. Under the $t(3)$  distribution with infinite fourth moment, the R-estimators yield smaller bias and significantly smaller MSE than the QMLE.

To strengthen the point that the R-estimators behave better than the QMLE under a heavy-tailed distribution, we have reported simulation results for larger sample sizes $n = 3000$ and $n = 5000$ under $t(3)$ distribution in Table \ref{sim.diff.n}. The QMLE failed to converge for large sample size; for example, with 
$n = 5000$ around 8\% replications do not converge. From Table \ref{sim.diff.n}, when $n$ increases, the  performance of the R-estimators becomes even better in terms of both the bias and MSE.

Overall, the vdW dominates the QMLE and other R-estimators sacrifice only small efficiency under the normal error distribution while they all achieve much higher efficiency when tails become much heavier. This provides a strong support for using the R-estimators. 

\begin{table}[!htbp]
\caption{The estimates of the standardized bias, MSE and ARE of  the R-estimators (sign, Wilcoxon and vdW) and the QMLE for the GARCH~($1, 1$) model under various error distributions with sample size $n = 1000$ based on $N = 500$ replications.}\vspace{-3mm}
\label{simulation1}
\begin{center}
\scriptsize  
\begin{tabular}{c c c c c c c c}\hline 
&\multicolumn{3}{c}{\textbf{Standardized bias}} &&\multicolumn{3}{c}{\textbf{Standardized MSE and ARE}} \\ \cline{2-4} \cline{6-8} & $\omega$ &$\alpha$&  $\beta$&  &$\omega$ &$\alpha$ &$\beta$ \\ \hline
\textbf{Normal} &&&&&&&  \\ 
QMLE	 & 8.96$\times 10^{-7}$	& -4.42$\times 10^{-4}$	& -1.54$\times 10^{-2}$	&&	6.45$\times 10^{-12}$&	1.41$\times 10^{-3}$&	4.14$\times 10^{-3}$ \\ 
Sign& 9.30$\times 10^{-7}$&	1.74$\times 10^{-3}$&	-1.54$\times 10^{-2}$&&		8.39$\times 10^{-12}$ (0.77)&	1.62$\times 10^{-3}$ (0.87)&	5.16$\times 10^{-3}$ (0.80) \\
Wilcoxon&	1.02$\times 10^{-6}$&	3.09$\times 10^{-3}$&	-1.61$\times 10^{-2}$&&	8.52$\times 10^{-12}$ (0.76)&	1.54$\times 10^{-3}$ (0.91)&	4.93$\times 10^{-3}$ (0.84) \\ 
vdW	& 9.05$\times 10^{-7}$ & 4.55$\times 10^{-4}$ &	-1.55$\times 10^{-2}$	&& 6.44$\times 10^{-12}$ (1.00) &	1.43$\times 10^{-3}$ (0.98)&	4.15$\times 10^{-3}$ (1.00)\\ \hline

\textbf{DE} &&&&&&&  \\ 
QMLE	& 1.02$\times 10^{-6}$&	3.56$\times 10^{-3}$&-2.26$\times 10^{-2}$&&	8.60$\times 10^{-12}$&	2.37$\times 10^{-3}$&	6.29$\times 10^{-3}$\\
Sign&	5.82$\times 10^{-7}$&	-3.42$\times 10^{-3}$&	-1.69$\times 10^{-2}$&&	6.22$\times 10^{-12}$ (1.38)&	1.74$\times 10^{-3}$ 	(1.36)&	5.15$\times 10^{-3}$ (1.22)\\
Wilcoxon&	6.24$\times 10^{-7}$&	-2.93$\times 10^{-3}$&	-1.74$\times 10^{-2}$&&	6.34$\times 10^{-12}$ (1.36)&	1.76$\times 10^{-3}$ (1.35) &	5.12$\times 10^{-3}$ (1.23) \\
vdW	&6.22$\times 10^{-7}$&	-4.13$\times 10^{-3}$&	-1.96$\times 10^{-2}$	&&	6.51$\times 10^{-12}$ (1.32) &	1.88$\times 10^{-3}$ (1.26)&	5.45$\times 10^{-3}$ (1.15)\\ \hline

\textbf{Logistic} &&&&&&& \\ 
QMLE& 1.05$\times 10^{-6}$&	2.51$\times 10^{-3}$& -1.51$\times 10^{-2}$	&&	7.44$\times 10^{-12}$&	1.63$\times 10^{-3}$&	4.28$\times 10^{-3}$\\
Sign&	6.85$\times 10^{-7}$&	-2.65$\times 10^{-3}$&	-1.17$\times 10^{-2}$&&		5.40$\times 10^{-12}$ (1.38) & 1.42$\times 10^{-3}$ (1.15) &	3.66$\times 10^{-3}$ (1.17)\\
Wilcoxon&	6.82$\times 10^{-7}$&	-2.91$\times 10^{-3}$&	-1.19$\times 10^{-2}$&&		5.24$\times 10^{-12}$ (1.42)&	1.38$\times 10^{-3}$ (1.18) &	3.56$\times 10^{-3}$ (1.20) \\
vdW	&7.06$\times 10^{-7}$&	-3.80$\times 10^{-3}$&	-1.34$\times 10^{-2}$&&		5.66$\times 10^{-12}$ (1.31)&	1.42$\times 10^{-3}$ (1.14)&	3.83$\times 10^{-3}$ (1.12)\\  \hline

${\boldsymbol {t(3)}}$ &&&&&&&\\ 
QMLE&	9.96$\times 10^{-7}$&	2.99$\times 10^{-2}$&	-5.46$\times 10^{-2}$&& 2.53$\times 10^{-11}$&	2.74$\times 10^{-2}$&	2.81$\times 10^{-2}$\\
Sign&	4.33$\times 10^{-7}$&	4.82$\times 10^{-3}$&	-1.80$\times 10^{-2}$&&		6.78$\times 10^{-12}$ (3.73)&	3.72$\times 10^{-3}$ 	(7.37)&	7.73$\times 10^{-3}$ (3.64)\\
Wilcoxon&	4.15$\times 10^{-7}$&	4.41$\times 10^{-3}$&	-1.83$\times 10^{-2}$&&		7.10$\times 10^{-12}$ (3.57) &	3.86$\times 10^{-3}$ (7.10) &	8.18$\times 10^{-3}$ (3.44)\\
vdW&	3.92$\times 10^{-7}$&	3.77$\times 10^{-3}$&	-2.57$\times 10^{-2}$&&		9.38$\times 10^{-12}$ (2.70)&	5.33$\times 10^{-3}$ (5.14) &	1.14$\times 10^{-2}$ (2.47)\\ \hline   
\end{tabular}
\end{center}
\end{table}

\begin{table}[!htbp]
\caption{The estimates of the bias, MSE and ARE of  the R-estimators (sign, Wilcoxon and vdW) and the QMLE for the GARCH~($1, 1$) model under the $t(3)$ error distribution with larger sample sizes $n = 3000, 5000$ based on $R = 500$ replications.}\vspace{-3mm}
\label{sim.diff.n}
\begin{center}
\scriptsize  
\begin{tabular}{cccccccc}
\hline
                  & \multicolumn{3}{c}{\textbf{Bias}} &  & \multicolumn{3}{c}{\textbf{MSE and ARE}} \\ \hline
                  & $\omega$       & $\alpha$      & $\beta$       &  & $\omega$         & $\alpha$         & $\beta$         \\ \hline
\textbf{n = 3000} &                &               &               &  &                  &                  &                 \\
QMLE              & 6.34$\times 10^{-7}$       & 1.80$\times 10^{-2}$      & -3.48$\times 10^{-2}$     &  & 1.14$\times 10^{-11}$         & 1.61$\times 10^{-2}$         & 1.25$\times 10^{-2}$        \\
Sign              & 1.52$\times 10^{-7}$       & 1.46$\times 10^{-3}$      & -9.99$\times 10^{-3}$     &  & 1.65$\times 10^{-12}$ (6.89)  & 1.29$\times 10^{-3}$ (12.47) & 2.10$\times 10^{-3}$ (5.93) \\
Wilcoxon          & 1.61$\times 10^{-7}$       & 1.47$\times 10^{-3}$      & -1.03$\times 10^{-2}$     &  & 1.76$\times 10^{-12}$ (6.46)  & 1.35$\times 10^{-3}$ (11.95) & 2.22$\times 10^{-3}$ (5.63) \\
vdW               & 1.58$\times 10^{-7}$       & 1.01$\times 10^{-3}$      & -1.39$\times 10^{-2}$     &  & 2.46$\times 10^{-12}$ (4.63)  & 1.89$\times 10^{-3}$ (8.49)  & 3.15$\times 10^{-3}$ (3.96) \\ \hline
\textbf{n = 5000} &                &               &               &  &                  &                  &                 \\ 
QMLE              & 3.66$\times 10^{-7}$       & 1.20$\times 10^{-2}$      & -2.07$\times 10^{-2}$     &  & 8.21$\times 10^{-12}$         & 1.20$\times 10^{-2}$         & 8.22$\times 10^{-3}$        \\
Sign              & 6.95$\times 10^{-11}$       & -2.00$\times 10^{-3}$     & -3.86$\times 10^{-3}$     &  & 1.01$\times 10^{-12}$ (8.09)  & 7.21$\times 10^{-4}$ (16.67) & 1.16$\times 10^{-3}$ (7.10) \\
Wilcoxon          & -3.01$\times 10^{-10}$      & -1.81$\times 10^{-3}$     & -3.98$\times 10^{-3}$     &  & 1.06$\times 10^{-12}$ (7.73)  & 7.56$\times 10^{-4}$ (15.90) & 1.20$\times 10^{-3}$ (6.85) \\
vdW               & -1.57$\times 10^{-8}$      & -2.37$\times 10^{-3}$     & -5.86$\times 10^{-3}$     &  & 1.54$\times 10^{-12}$ (5.33)  & 1.13$\times 10^{-3}$ (10.64) & 1.77$\times 10^{-3}$ (4.64)  \\ \hline
\end{tabular}
\end{center}
\end{table}

\textbf{Simulation for the GARCH~($2, 1$) model.}
It was reported in Francq and Zako{\"\i}an (2009) that higher order GARCH models may fit some financial time series better than the GARCH~($1, 1$) model. Therefore, here we examine the performance of the R-estimators under the GARCH~($2, 1$) model by running simulations with $R = 500, n = 1000$. To choose the true model parameter for simulation, we fitted the FTSE 100 data from January 2007 to December 2009 to the by GARCH~($2, 1$) model using the {\tt fGarch} package. It turned out that 
$\alpha_2$ is significant with $p$-value $=0.019$ and the Akaike information criterion (AIC) of the GARCH~($2, 1$) is smaller than that of the GARCH~($1, 1$). Since the {\tt fGarch} estimate of the true parameter
is $\bth_0 = (4.46\times 10^{-6}, 0.0525, 0.108, 0.832)^\prime$, we choose this $\bth_0$ to generate sample from the GARCH~($2, 1$) model with various error distributions. The R-estimators and QMLE are compared through the  bias and MSE and  the corresponding estimates are reported in Table \ref{sim.GARCH21}. Similar to the GARCH~($1, 1$) case, the advantage of the R-estimators over the QMLE becomes prominent under heavy-tailed distributions, especially under the $t(3)$ distribution, where the bias and MSE of the R-estimators have smaller order of magnitude than the those of the QMLE. 

\begin{table}[!htbp]
\caption{The estimates of the bias and MSE  of  the R-estimators (sign, Wilcoxon and vdW scores) and the QMLE for the  GARCH~($2, 1$) model under various error distributions (sample size $n = 1000$; $R = 500$ replications).}\vspace{-3mm}
\label{sim.GARCH21}
\begin{center}
\scriptsize  
\begin{tabular}{cccccccccc}
\hline
                  & \multicolumn{4}{c}{\textbf{Bias}}          &  & \multicolumn{4}{c}{\textbf{MSE}}  \\ \hline
                  & $\omega$ & $\alpha_1$ & $\alpha_2$ & $\beta$   &  & $\omega$ & $\alpha_1$ & $\alpha_2$ & $\beta$  \\ \hline
\textbf{Normal}   &          &            &            &           &  &          &            &            &          \\
QMLE              & 3.80$\times 10^{-6}$ & 8.85$\times 10^{-3}$   & -3.16$\times 10^{-3}$  & -2.01$\times 10^{-2}$ &  & 2.50$\times 10^{-11}$ & 1.71$\times 10^{-3}$   & 1.93$\times 10^{-3}$   & 1.35$\times 10^{-3}$ \\
Sign              & 3.79$\times 10^{-6}$ & 1.05$\times 10^{-2}$   & -5.76$\times 10^{-3}$  & -1.84$\times 10^{-2}$ &  & 2.65$\times 10^{-11}$ & 1.90$\times 10^{-3}$   & 2.19$\times 10^{-3}$   & 1.30$\times 10^{-3}$ \\
Wilcoxon          & 3.68$\times 10^{-6}$ & 9.91$\times 10^{-3}$   & -6.51$\times 10^{-3}$  & -1.81$\times 10^{-2}$ &  & 2.42$\times 10^{-11}$ & 1.74$\times 10^{-3}$   & 2.01$\times 10^{-3}$   & 1.20$\times 10^{-3}$ \\
vdW               & 3.95$\times 10^{-6}$ & 1.03$\times 10^{-2}$   & -7.49$\times 10^{-3}$  & -1.96$\times 10^{-2}$ &  & 2.67$\times 10^{-11}$ & 1.74$\times 10^{-3}$   & 1.94$\times 10^{-3}$   & 1.25$\times 10^{-3}$ \\ \hline
\textbf{DE}       &          &            &            &           &  &          &            &            &          \\
QMLE              & 2.61$\times 10^{-6}$ & 4.43$\times 10^{-3}$   & 2.14$\times 10^{-3}$   & -1.99$\times 10^{-2}$ &  & 3.11$\times 10^{-11}$ & 2.53$\times 10^{-3}$   & 4.01$\times 10^{-3}$   & 2.33$\times 10^{-3}$ \\
Sign              & 1.96$\times 10^{-6}$ & 5.02$\times 10^{-3}$   & 2.57$\times 10^{-4}$   & -1.61$\times 10^{-2}$ &  & 9.96$\times 10^{-12}$ & 1.85$\times 10^{-3}$   & 2.88$\times 10^{-3}$   & 1.66$\times 10^{-3}$ \\
Wilcoxon          & 1.85$\times 10^{-6}$ & 3.03$\times 10^{-3}$   & -2.03$\times 10^{-3}$  & -1.65$\times 10^{-2}$ &  & 9.57$\times 10^{-12}$ & 1.79$\times 10^{-3}$   & 2.85$\times 10^{-3}$   & 1.73$\times 10^{-3}$ \\
vdW               & 1.95$\times 10^{-6}$ & 1.80$\times 10^{-3}$   & -1.96$\times 10^{-3}$  & -1.81$\times 10^{-2}$ &  & 1.10$\times 10^{-11}$ & 1.92$\times 10^{-3}$   & 3.14$\times 10^{-3}$   & 1.97$\times 10^{-3}$ \\ \hline
\textbf{Logistic} &          &            &            &           &  &          &            &            &          \\
QMLE              & 4.72$\times 10^{-6}$ & 5.44$\times 10^{-3}$   & 8.41$\times 10^{-4}$   & -1.98$\times 10^{-2}$ &  & 5.24$\times 10^{-11}$ & 3.75$\times 10^{-3}$   & 4.49$\times 10^{-3}$   & 2.06$\times 10^{-3}$ \\
Sign              & 3.17$\times 10^{-6}$ & 3.23$\times 10^{-3}$   & -2.32$\times 10^{-3}$  & -1.49$\times 10^{-2}$ &  & 2.09$\times 10^{-11}$ & 1.75$\times 10^{-3}$   & 2.50$\times 10^{-3}$   & 1.39$\times 10^{-3}$ \\
Wilcoxon          & 3.24$\times 10^{-6}$ & 2.93$\times 10^{-3}$   & -1.97$\times 10^{-3}$  & -1.51$\times 10^{-2}$ &  & 2.20$\times 10^{-11}$ & 1.73$\times 10^{-3}$   & 2.48$\times 10^{-3}$   & 1.42$\times 10^{-3}$ \\
vdW               & 3.62$\times 10^{-6}$ & 2.49$\times 10^{-3}$   & -1.97$\times 10^{-3}$  & -1.72$\times 10^{-2}$ &  & 2.76$\times 10^{-11}$ & 1.91$\times 10^{-3}$   & 2.67$\times 10^{-3}$   & 1.72$\times 10^{-3}$ \\ \hline
$\mbf{t(3)}$   &          &            &            &           &  &          &            &            &          \\
QMLE              & 1.78$\times 10^{-6}$ & 3.06$\times 10^{-2}$   & -2.07$\times 10^{-2}$  & -3.12$\times 10^{-2}$ &  & 2.85$\times 10^{-11}$ & 7.88$\times 10^{-2}$   & 7.65$\times 10^{-2}$   & 1.08$\times 10^{-2}$ \\
Sign              & 9.92$\times 10^{-7}$ & 3.18$\times 10^{-3}$   & -3.92$\times 10^{-3}$  & -1.29$\times 10^{-2}$ &  & 5.67$\times 10^{-12}$ & 3.25$\times 10^{-3}$   & 5.25$\times 10^{-3}$   & 2.42$\times 10^{-3}$ \\
Wilcoxon          & 9.78$\times 10^{-7}$ & 3.69$\times 10^{-3}$   & -4.87$\times 10^{-3}$  & -1.28$\times 10^{-2}$ &  & 5.70$\times 10^{-12}$ & 3.51$\times 10^{-3}$   & 5.58$\times 10^{-3}$   & 2.50$\times 10^{-3}$ \\
vdW               & 9.86$\times 10^{-7}$ & 5.10$\times 10^{-3}$   & -9.49$\times 10^{-3}$  & -1.56$\times 10^{-2}$ &  & 7.59$\times 10^{-12}$ & 5.66$\times 10^{-3}$   & 8.08$\times 10^{-3}$   & 3.57$\times 10^{-3}$ \\ \hline
\end{tabular}
\end{center}
\end{table}

\section{Bootstrapping the R-estimators}\label{sec.Boot}
Since the asymptotic covariance matrices of the R-estimators are of complicated forms, in this section we employ the weighted bootstrap technique discussed by Chatterjee and Bose~(2005) in the context of M-estimators to approximate the distributions of the R-estimators and we compute corresponding coverage probabilities to exhibit the effectiveness of such bootstrap approximations.  
The weighted bootstrap in this context is attractive for its computational simplicity since at each bootstrap replication, only the weights need to be generated instead of resampling the data components to compute the replicates of the bootstrapped R-estimate. 

In this context, the weighted bootstrap version of the rank-based central sequence is
$$\h{\mbf R}_{n, \vp}^* (\bth) := \h{\mbf R}_n^* (\bth) = n^{-1/2} \sum_{t=1}^n w_{nt} \frac{\d{\h{\mbf v}}_t(\bth)}{\h{v}_t(\bth)}\left\lbrace 1-\vp\left[\frac{\h{R}_{nt}(\bth)}{n+1}\right]\frac{X_t}{\h{v}_t^{1/2}(\bth)}\right\rbrace,$$
where $\{w_{nt}; 1\leq t \leq n; n \geq 1\}$ is a triangular array of r.v.'s which satisfies the following conditions:\\
(i) The weights $\{w_{nt}; 1\leq t \leq n\}$ are exchangeable and independent of the data $\{X_t; 1\leq t \leq n\}$ and errors $\{\epsilon_t;  1\leq t \leq n \};$ \\
(ii) For all $t \geq 1, w_{nt} \geq 0; {\mathrm E}(w_{nt}) = 1; \text{Corr}(w_{n1}; w_{n2}) = O(1/n); \text{Var}(w_{nt}) = \sigma_n^2,$ where $0 < c_1 < \sigma_n^2 = o(n),$ with $c_1 >0$ being a constant.

Among various schemes of the weights satisfying the above conditions, we compare the following three types of weights:\\
(i) Scheme M: $\{w_{n1}, \ldots, w_{nn}\}$ have a multinomial $(n, 1/n, \ldots, 1/n)$ distribution, which is essentially the
classical paired bootstrap. \\
(ii) Scheme E: $w_{nt}=(n E_t)/\sum_{i=1}^n E_i$, where $\{E_t\}$ are i.i.d. exponential r.v.'s with mean $1$. \\ 
(iii) Scheme U: $w_{nt}=(n U_t)/\sum_{i=1}^n U_i$,
where $\{U_t\}$ are i.i.d. uniform r.v.'s on $(0.5, 1.5)$.

We propose the following Algorithm~\ref{algBootRestGARCH} to compute our bootstrap estimator, where the weighted version of (\ref{eq:formula}) is used to compute the bootstrap estimator $\h{\bth}_{*n\vp}$ of $\h{\bth}_{n\vp}$ and then with $\hat{c}_{\vp}$ given in \eqref{hatcvp}, we obtain the bootstrap estimator $\h{\bth}_{*n}$ of $\h{\bth}_{n}$ through multiplying the $\omega$ and $\alpha$ components by $\hat{c}_{\vp}^{-1}$.

\begin{algorithm}[!htbp]
\SetAlgoLined\vspace{3mm}

\KwIn{a sample $\{X_t; 1 \leq t \leq n\}$, orders $p$ and $q$ of the GARCH process, numbers $k$ and $k^*$ of iterations in the one-step and bootstrap procedures respectively.}
\KwOut{Bootstrap estimator $\hat{\bth}_{*n}$}
\begin{enumerate}

\item Compute the R-estimator $\h{\bth}_{n\vp}$ using \eqref{eq:formula} and set $\h{\bth}_{*n\vp} = \h{\bth}_{n\vp}$.
 
\item \For{$i \gets 1$ \textbf{to} $k^*$}{

\begin{equation}
\begin{split}
\h{\bth}_{*n\vp} &\gets \h{\bth}_{*n\vp}-\left[\sum_{t=1}^n \frac{\d{\h{\mbf v}}_t(\h{\bth}_{*n\vp})\d{\h{\mbf v}}_t^{'}(\h{\bth}_{*n\vp})}{\h{v}_t^2(\tilde{\bth}_{*n\vp})}\right]^{-1} \\
 & \quad \quad \quad \quad \times\left\lbrace\sum_{t=1}^n \frac{\d{\h{\mbf v}}_t(\h{\bth}_{*n\vp})}{\h{v}_t(\tilde{\bth}_{*n\vp})}\left[1-\varphi\left(\frac{R_{nt}(\h{\bth}_{*n\vp})}{n+1}\right)\frac{X_t}{\h{v}_t^{1/2}(\h{\bth}_{*n\vp})}\right]\right\rbrace
\end{split} 
\label{eq:formulaBoot}
\end{equation}
}

\item Compute  $\hat{c}_{\vp}$ using \eqref{hatcvp}, and then compute $\h{\bth}_{*n}$ through multiplying the $\omega$ and $\alpha$ components by $\hat{c}_{\vp}^{-1}$.

\end{enumerate}

\caption{Bootstrapping R-estimator for GARCH models}\label{algBootRestGARCH}
\end{algorithm}

%

\subsection{Bootstrap coverage probabilities}

Chatterjee and Bose~(2005) proved the consistency of the bootstrap for an estimator defined by smooth estimating equation. 
Since ranks are integer-valued discontinuous functions, the proof of the asymptotic validity of the bootstrapped R-estimator is a mathematically challenging problem which is beyond the scope of this paper. Instead, we resort to simulations to evaluate the performance of the bootstrap approximation of the R-estimators by comparing the distribution of $\sigma_n^{-1} \sqrt{n} (\h{\bth}_{*n} - \h{\bth}_{n})$ with that of $\sqrt{n} (\h{\bth}_{n} - \bth_{0})$ in terms of coverage rates. 

In particular, with the choice of the true parameter $\bth_0 = (6.50\times 10^{-6}, 0.177, 0.716)^\prime$ as in the simulation study of the GARCH~($1, 1$) model of the previous section, we generate $R=1000$ data each with sample size $n = 1000$ based on different error distributions. For each data, the exchangeable weights $\{w_{nt}; 1\leq t \leq n\}$ are generated $B=2000$ times. We consider cases where the error distributions are normal, DE, logistic and $t(3)$. The bootstrap weights are based on Schemes M, E and U. 
The bootstrap coverage rates (in percentage) for $95\%, 90\%$ nominal levels are reported in Table \ref{bootstrap1}. Notice that all bootstrap schemes provide reasonable coverage rates under these error distributions. Scheme U is slightly better than the scheme M and E under the DE and $t(3)$ distributions. 

\begin{table}
\caption{The bootstrap coverage rates (in percentage) for the R-estimators (sign, Wilcoxon and vdW) under various error distributions}
\label{bootstrap1}
\begin{center}
\begin{tabular}{c c c c c c c c c c} \hline
&&&\multicolumn{3}{c}{$95\%$ nominal level} &&\multicolumn{3}{c}{$90\%$ nominal level} \\ \cline{4-6} \cline{8-10} & &&$\omega$ &$\alpha$&$\beta$ & & $\omega$ &$\alpha$&$\beta$ \\ \hline
Normal& sign & Scheme M &  94.1&	93.6&	92.8  &&	90.8 &	88.5 &	89.2  \\
&&Scheme E &  93.5  &	93.3 	&92.8  &&		90.3 &	88.1  &	88.8  \\
&&Scheme U & 94.8  &	94.7 &	93.7 &&	91.7 &	90.2 &	89.6 
\\  \hline
Normal& Wilcoxon &Scheme M & 96.4 &	96.3 &	93.7 &&	93.8 &	90.3 &	89.7  \\
&&Scheme E & 96.5 &	96.0 &	93.2 &&		93.2 &	90.0 &	89.3  \\
&&Scheme U & 96.5  &	95.6 &	94.3  &&	93.5 &	91.1  &	90.1  \\ \hline
Normal & vdW &Scheme M & 94.3 &	92.3 	&93.6  &&		91.3 &	89.1  &	89.0  \\
&&Scheme E &  94.2 &	92.2 &	92.7  &&	90.5  &	88.5  &	88.7 \\
&&Scheme U & 95.3  &	94.1  &	93.7  &&	91.4  &	90.5  &	89.2  \\ \hline
DE & sign & Scheme M &  90.8 &	90.5  &	91.6  &&	87.8  &	86.0  &	86.8  \\
&&Scheme E & 90.4  &	89.6  &	90.4  &&		87.2  &	85.3 &	86.4   \\
&&Scheme U & 91.9 	&92.7  &	92.7  &&		88.8  &	88.9  &	87.8   \\ \hline
DE & Wilcoxon &Scheme M & 91.0 	&91.0 &	91.5  &&		87.6 &	86.9  &	87.6  \\
&&Scheme E & 90.7  &	90.2 	&90.4  &&		87.2  &	86.2  &	86.6 \\
&&Scheme U & 92.4  &	93.6  &	92.8  &&		88.7  &	89.1  &	87.7 \\ \hline
DE & vdW &Scheme M & 90.9 &	87.6 	&89.7  &&		87.5 &	83.9 &	85.3  \\
&&Scheme E &  90.4  &	86.9  &	88.9  &&	86.9  &	83.1  &	84.8 \\
&&Scheme U & 92.4  &	90.2  &	91.0  &&		89.7  &	85.6  &	86.0  \\ \hline
Logistic & sign & Scheme M & 93.0  &	91.1  &	92.1  &&		89.0  &	87.6  &	88.6  \\ 
&&Scheme E &  93.4  &	92.3  &	92.5  &&		89.8  &	86.3  &	88.4  \\
&&Scheme U &  93.0  &	92.3 	& 91.9  &&		88.7 &	87.5  &	87.1  \\ \hline
Logistic & Wilcoxon &Scheme M & 93.5  &	91.3  &	92.5  &&		89.9  &	87.7 	& 89.2  \\
&&Scheme E & 93.7  &	89.4  &	91.7  &&		90.0  &	85.8  &	87.1 \\
&&Scheme U & 94.1  &	92.2  &	92.8  &&		88.9  &	88.1 &	86.4  \\ \hline
Logistic & vdW &Scheme M & 93.1  &	91.2  &	92.3  &&		89.3 &	88.0 &	87.0  \\
&&Scheme E & 92.4  &	91.1  &	91.7  &&		88.5  &	87.6  &	86.4 \\
&&Scheme U & 94.4  &	93.6  &	92.2  &&		90.4  &	90.8  &	86.8  \\ \hline
$t(3)$ & sign & Scheme M &  88.3  &	85.3  &	88.3  &&		86.0  &	82.6  &	83.5   \\ 
&&Scheme E &  88.3 	&85.0  &	87.6  &&		84.9  &	82.4  &	82.0  \\
&&Scheme U &  91.8  &	89.0  &	90.6  &&		87.5  &	85.6  &	86.4   \\ \hline
$t(3)$ & Wilcoxon &Scheme M & 88.1 &	84.7  &	88.5  &&		85.7  &	81.4  &	83.7  \\
&&Scheme E & 88.0 	& 84.5  &	87.7  &&		85.0  &	80.7  &	82.8 \\
&&Scheme U & 91.8  &	88.7  &	90.0  &&		87.6  &	85.6  &	85.5 \\ \hline
$t(3)$ & vdW &Scheme M & 85.6  &	82.3  &	86.1  &&		81.9  &	79.7  &	81.0  \\
&&Scheme E & 84.4  &	82.1  &	86.0  &&		80.9  &	78.7  &	80.2 \\
&&Scheme U & 90.3  &	85.4  &	88.9  &&		86.6  &	81.0  &	83.3  \\ \hline
\end{tabular}
\end{center}
\end{table}

To check the performance of the bootstrap under different sample sizes, we run simulation with $n = 200, 300, ..., 1000$ for the sign, Wilcoxon and vdW scores. There are $R=1000$ replications being generated under the normal error distribution, and each replication is bootstrapped $B=2000$ times with the scheme U. Figure~\ref{fig:BootR} shows the bootstrap coverage rates for $\omega$ 
(first row), $\alpha$  (second row), $\beta$  (third row) under $95\%$ nominal level (left column) and $90\%$ nominal level (right column). We notice that as the sample size increases, the coverage rates get close to the nominal levels for all parameters and all R-estimators, with only few exceptions. This tends to imply the consistency of the bootstrap approximation. With the sample size $n \geq 500$, the bootstrap coverage rates are generally close to the nominal levels.

\begin{figure}
\begin{center}
\caption{Plot of the bootstrap coverage rates for the  R-estimators (sign, Wilcoxon and vdW) at different sample sizes. The first, second and third rows are for $\omega, \alpha$ and $\beta$ respectively. The nominal levels are $95\%$ (left column) and $90\%$ (right column). Scheme U is employed and the errors have normal distribution.}
\includegraphics[width=6.5in]{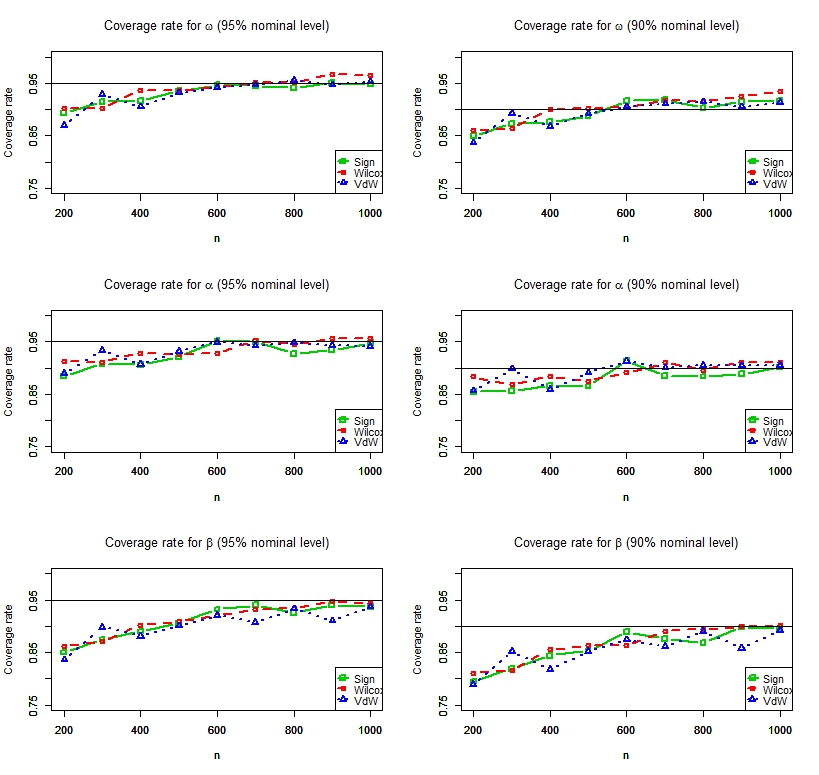}\label{fig:BootR}
\end{center}
\end{figure}

\section{Application of the R-estimator to the GJR model}\label{sec.GJR.R.main}
The GJR model, proposed by Glosten et al.~(1993), is used frequently for asymmetric financial data.
Iqbal and Mukherjee~(2010) considered a class of M-estimators to estimate model parameters. In a similar fashion, we have used the new class of R-estimators to analyze the GJR model and the relevant simulation results are available in Appendix~
\ref{sec.GJR}.


\section{Conclusion}\label{sec.con.R.GARCH}
We propose a new class of R-estimators for the GARCH model and derive the asymptotic normality of these estimators under mild moment and smoothness conditions on the error distribution. We exhibit the robustness and efficiency of R-estimators with respect to the QMLE through simulation and real data analysis. We also consider a general type of weighted bootstrap for the R-estimators which is computational-friendly and easy-to-implement. The theoretical analysis such as the asymptotic validity of the weighted bootstrap is an interesting but challenging problem that can be explored in the future. 

\section*{Acknowledgements}
 Hang Liu gratefully acknowledges the \textit{ESRC North West Social Science Doctoral Training Partnership} (NWSSDTP Grant Number ES/P000665/1) for funding his PhD studentship.









\pagebreak


\newcommand{\beginsupplement}{%
        \setcounter{table}{0}
        \renewcommand{\thetable}{S\arabic{table}}%
        \setcounter{figure}{0}
        \renewcommand{\thefigure}{S\arabic{figure}}%
        \setcounter{equation}{0}
        \renewcommand{\theequation}{S\arabic{equation}}%
        \setcounter{section}{0}
        \renewcommand{\thesection}{S\arabic{section}}
        \renewcommand{\thetheorem}{S\arabic{theorem}}
     }


\appendix

\newpage 

\begin{center}
\textbf{\large APPENDIX}
\end{center}

\setcounter{page}{1}

\section{Proofs of Proposition~\ref{Prop.asy.linear}, Theorem~\ref{thm.asy} and \ref{thm.asybth0}}\label{sec.proof}
We will use the following facts from Berkes et al.~(2003) for the proofs:

Fact 1. For any $\nu > 0$,
\begin{equation}\label{fact1}
{\mathrm E}\left\lbrace \sup \left[\left| \frac{\d{\mbf v}_1(\bth)}{v_1(\bth)}\right|^\nu; \bth \in \bTH_0 \right]\right\rbrace < \infty.
\end{equation}
and
\begin{equation*}
{\mathrm E}\left\lbrace \sup \left[\left| \frac{\ddot{\mbf v}_1(\bth)}{v_1(\bth)}\right|^\nu; \bth \in \bTH_0 \right]\right\rbrace < \infty.
\end{equation*}

Fact 2. There exist random variables $Z_0$, $Z_1$ and $Z_2$, all independent of $\{\epsilon_t; 1\leq t \leq n\}$ and a number $0 < \rho < 1$, such that
\begin{equation}\label{fact21}
0 < v_t(\bth) - \h{v}_t(\bth) \leq \rho^t Z_0,
\end{equation}
\begin{equation}\label{fact22}
|\d{\mbf v}_t(\bth) - \d{\h{\mbf v}}_t(\bth)| \leq \rho^t Z_1,
\end{equation}
\begin{equation*}
|\ddot{\mbf v}_t(\bth) - \ddot{\h{\mbf v}}_t(\bth)| \leq \rho^t Z_2.
\end{equation*}

Fact 3. Let $\{(A_t, B_t, C_t); t \geq 0\}$ be a sequence of identically distributed random variables. If ${\mathrm E} \log^+ A_0 + {\mathrm E} \log^+ B_0 + {\mathrm E} \log^+ C_0 < \infty$, then for any $|r| < 1$,
\begin{equation}\label{fact3}
\sum_{t=0}^\infty (A_t + B_t C_t) r^t \quad \text{converges with probability 1}.
\end{equation}

\textbf{Idea of the proof of Theorem 
\ref{thm.asy}.} We first derive the following Theorem \ref{Thm.emp}, Corollary \ref{Corollary} and Theorem \ref{UUbb} on empirical processes where a scale-perturbed weighted mixed-empirical process is approximated by its non-perturbed version. With 
$\bth_{n\vp} = \bth_{0\vp} + n^{-1/2} \bb$, we derive asymptotic expansion of the difference between two quantities $\T_{1n}(\bth_{n\vp})$ and $\T_{2n}(\bth_{n\vp})$ which are defined later. We then show that $\T_{1n}(\bth_{n\vp})$ can be approximated by a r.v., which is asymptotic normal, plus a term linear in $\bb$. Also, we use  $\T_{2n}(\bth_{n\vp})$ to approximate ${\mbf R}_n(\bth_{n\vp})$ and show that asymptotically their difference is a r.v. with mean zero. Finally, we prove that the difference of ${\mbf R}_n(\bth_{n\vp})$ and $\h{{\mbf R}}_n(\bth_{n\vp})$ converges in probability to zero. Using these results, we are able to derive the asymptotic linearity of $\h{{\mbf R}}_n(\bth_{n\vp})$ as shown in Proposition 
\ref{Prop.asy.linear}. Finally, using the definition of the one-step R-estimator in (
\ref{onestep.def}), we are able to derive the  asymptotic distribution of $\h{\bth}_{n\vp}$.\\


\textbf{Theorem \ref{Thm.emp}, Corollary \ref{Corollary} and Theorem \ref{UUbb}.}

Let $\{(\eta_t,\,\g_{nt},\del_{nt}),\,1\le t\le n\}$ be an array of
$3$-tuple r.v.'s defined on a probability space such that $\{\eta_t,\,1\le t\le n\}$ are  i.i.d. with c.d.f. $G$
and $\eta_t$ is independent of $(\g_{nt},\del_{nt})$ for each $1\le t\le n$. Let $\{{\cal A}_{nt}; 1 \le t \le n\}$ be an array of increasing sub-$\sigma$-fields in both $n$ and $t$ so that ${\cal A}_{nt} \subset {\cal A}_{n(t+1)}$, 
${\cal A}_{nt} \subset {\cal A}_{(n+1)t}$, $1 \le t \le n-1$, $n \ge 2$. Assume also that $(\g_{n1},\del_{n1})$ is ${\cal A}_{n1}$ measurable, and 
$\left\{\{(\g_{nt},\del_{nt}); 1 \le t \le j\}, \eta_1, \eta_2, \ldots, \eta_{j-1}\right\}$ are
${\cal A}_{nj}$ measurable, $2 \le j \le n$. For $x\in \R$, recall that $\mu(x)={\mathrm E}[\eta I(\eta<x)]=\int_{-\infty}^x sg(s)ds$
and consider the following weighted mixed-empirical processes 
\benr\label{v}
&&\ti V_n(x):= n^{-1/2}\s \g_{nt} \eta_t I(\eta_t< x+x\del_{nt}),\\ 
&& \ti J_n(x):= n^{-1/2}\s \g_{nt} \mu(x+x \del_{nt}),\nn \\ 
&&V^*_n(x):= n^{-1/2}\s \g_{nt} \eta_t I(\eta_t\le x),\qquad \qquad\,\, J^*_n(x):= n^{-1/2}\s \g_{nt} \mu(x ),\nn \\ 
&&  \ti U_n(x):= \ti V_n(x)-\ti J_n(x), \quad U^*_n(x):= V^*_n(x)-J^*_n(x).\nn
\eenr
Assume the following conditions on the weights $\{\g_{nt}\}$ and perturbations $\{\del_{nt}\}$.

Let $C_n:=\s {\rm E}|\gamma_{nt}|^q$ for some $q > 2$. Let $a$ with $0<a<q/2$ be such that
\benr
&& C_n/n^{q/2-a}=o(1) \label{c}.\\
&& \left(n^{-1}\s \g_{nt}^2\right)^{1/2} =\g+o_{\rm P}(1) 
\,\,\mbox{for a positive r.v.} \g. \label{ga1}\\
&& {\mathrm E}\left(n^{-1}\s \g_{nt}^2 \right)^{q/2} =  O(1).\label{bga2}\\
&& \max_{1\le t\le n}n^{-1/2}|\g_{nt}|= o_{\rm P}(1).\label{ga2}\\
&& \max_{1\le t\le n}|\del_{nt}|=o_{\rm P}(1)\label{ga3}.\\
&& \frac{n^{q/2-\ep}}{C_n}{\mathrm E}\left[n^{-1}\s \{\g_{nt}^2|\del_{nt}| \}\right]^{q/2} = o(1)\label{bga3}. \\
&& n^{-1/2}\s |\g_{nt}\del_{nt}|=O_{\rm P}(1).\label{ga4}
\eenr
The following theorem shows that {\it uniformly} over the entire real line, the perturbed process 
$\ti U_n$ can be approximated by $U^*_n$. 

\begin{theorem}\label{Thm.emp}
Under the above set-up and Assumptions (\r{c})-(\r{ga4}) and (A1),
\benr \sup_{x\in \R} |\ti U_n(x)-U^*_n(x)|&=&o_{\rm P}(1). \label{ti*} 
\eenr
\end{theorem}

\begin{proof}
The proof is similar to the proof in Mukherjee (2007, Theorem 6.1). In particular, we show point-wise convergence for each $x$ and then invoke the monotone structure of the mean processes to achieve the uniform convergence. For weighted empirical, the monotonically increasing mean process is given by the distribution function. Although $\mu$ in the present case is not a monotone function on $(-\infty, \infty)$, we use its monotone property separately on $(-\infty, 0]$ and $[0, \infty)$. 
\end{proof}

We remark that this theorem is different from Koul and Ossiander (1994, Theorem 1.1) and Mukherjee (2007, Theorem 6.1) where weighted empirical processes were considered for the estimation of the mean parameters. For the estimation of the scale parameters, in this paper we consider weighted mixed-empirical process which is a weighted sum of the mixture of error and its indicator process. 
 

The following corollary describes a Taylor-type expansion of the weighted sum of indicator functions $\ti V_n(x)$.

\begin{corollary}\label{Corollary}
Under the above setup and under the Assumptions (\r{c})-(\r{ga4}) and (A1),
\begin{equation}\label{JJ*}
\sup_{x\in \R} |\ti J_n(x)-J^*_n(x)-x^2 g(x) n^{-1/2}\s \g_{nt}\del_{nt}|=o_{\rm P}(1). 
\end{equation}
Hence,
\begin{equation}\label{VV*}
\sup_{x\in \R} |\ti V_n(x)-V_n^*(x)-x^2 g(x) n^{-1/2}\s \g_{nt}\del_{nt}|=o_{\rm P}(1). 
\end{equation}
\end{corollary}
\begin{proof}
Here (\ref{VV*}) follows from (\ref{JJ*}) and (\ref{ti*}). Therefore, it remains to prove (\ref{JJ*}). 
Notice that the LHS of (\ref{JJ*}) equals
\begin{align*}
& \sup_{x\in \R} \left|n^{-1/2}\s \gamma_{nt} \left[ x \int_{x}^{x + x\del_{nt}}   s g(s) ds  - x^2 g(x) \delta_{nt} \right]  \right| \\
=& \sup_{x\in \R} \left|n^{-1/2}\s \gamma_{nt}  \delta_{nt} \left[ x \int_{0}^{1}   (x + hx \delta_{nt}) g(x + hx \delta_{nt}) dh - x^2 g(x) \right]  \right| \\
= & o_{\rm P} (1)
\end{align*}
due to (\ref{ga4}) and Assumption (A1).
\end{proof}

The next theorem provides an extended version of (\r{ti*}) when the weights are functions on appropriately scaled parameter space. We define the following processes of two arguments as follows.
  
{\bf Probabilistic framework}: Let $\{\eta_t,\,1\le t\le n\}$ be  i.i.d. with the c.d.f. $G$, $\{l_{nt};1\le t \le n\}$ be an array of measurable  functions from $\R^m$ to $\R$ such that for every $\bb\in \R^m$ and $1\le t\le n$, $(l_{nt}(\bb), \, u_{nt}(\bb))$ are
independent of $\eta_t$. For $x\in \R$ and $\bb \in \R^m$, let
\benrr
&& \ti{\cal V}(x,\bb):= n^{-1/2}\s \, l_{nt}(\bb) \eta_t I\Big(\eta_t< x+xu_{nt}(\bb)\Big),\\ &&
\ti {\cal J}(x,\bb):= n^{-1/2}\s \, l_{nt}(\bb)\,\mu\Big(x+xu_{nt}(\bb)\Big),\nn \\
&&\ti {\cal U}(x,\bb):= \ti{\cal V}(x,\bb)- \ti{\cal J}(x,\bb), \nn \\
&& {\cal V}^*(x,\bb):= n^{-1/2}\s \, l_{nt}(\bb) \eta_t I(\eta_t< x), \, {\cal J}^*(x,\bb):= n^{-1/2}\s \, l_{nt}(\bb)\,\mu(x),\nn \\
&& {\cal U}^*(x,\bb):= {\cal V}^*(x,\bb)- {\cal J}^*(x,\bb)=n^{-1/2}\s l_{nt}(\bb) \Big[\eta_t I(\eta_t < x)-\mu(x)\Big]. 
\eenrr
Here ${\cal U}^*(\cdot,\cdot)$ is a sequence of ordinary non-perturbed weighted mixed-empirical processes with weights $\{l_{nt}(\cdot)\}$ and $\ti {\cal U}(\cdot, \cdot)$ is a sequence of perturbed weighted mixed-empirical processes with scale perturbations $\{u_{nt}(\cdot)\}$. In Theorem \ref{UUbb} below it is shown that 
$\ti {\cal U}$ can be uniformly approximated by ${\cal U}^*$ under the following conditions (\r{cnt})-(\r{ll}) for $\{l_{nt}(\cdot)\}$ and $\{u_{nt}(\cdot)\}$. Note that the statements on assumptions and convergence hold point-wise for each fixed $\bb\in \R^m$. 

There exist numbers $q > 2$ and $a$ (both free from $\bb$) satisfying $0<a<q/2$ such that with $C_n(\bb):=\s {\rm E}|l_{ni}(\bb)|^q$, 
\begin{equation}\label{cnt}
C_n(\bb)/n^{q/2-a}=o(1), \,\,\, \mbox{for each}\,\,\, \bb \in \R^m.
\end{equation}
For some positive random process $\ell(\bb)$,
\benr &&\left(n^{-1}\s \,l^2_{nt}(\bb)\right)^{1/2}
=\ell(\bb)+o_{\rm P}(1),\hskip 1in \bb\in \R^m.\label{l1}\\
&&  {\mathrm E}\left(n^{-1}\s l_{ni}^2(\bb) \right)^{q/2}=O(1),\hskip 1.55in
\bb\in \R^m.\label{11*}\\
&&  \max_{1\le t\le n} n^{-1/2}|l_{nt}(\bb)|=o_{\rm P}(1),\hskip 1.55in \bb\in \R^m. \label{l2}\\
&& \max_{1\le t\le n} \{|u_{nt}(\bb)|\}=o_{\rm P}(1),\hskip 1.1in
\bb\in \R^m.\label{g3}\\
&&  \frac{n^{q/2-a}}{C_n(\bb)}{\mathrm E}\left[ n^{-1} \s
l_{nt}^2(\bb)|u_{nt}(\bb)|\right]^{q/2}=o(1),\hskip 1.0in
\bb\in \R^m.\label{g5*}\\
&&  n^{-1/2}\s \,l_{nt}(\bb)u_{nt}(\bb)=O_{\rm P}(1), \quad \bb\in \R^m.\label{g5}\\
&& \forall\, \, b \, \mbox{and}\, \varepsilon>0, \,\exists\,
\del>0,\,\,\mbox{and}\,\, n_1 \in \NN \, \mbox{whenever}\, \|\bs\|\le
b,\, \mbox{and}\,
 n>n_1,\label{kk} \\
&& P\Bigg(n^{-1/2}\s
\,|l_{nt}(\bs)|\,\Big\{\sup_{\|\bt-\bs\|<\del}|u_{nt}(\bt)- u_{nt}(\bs)|
 \Big\}\le \varepsilon\Bigg)>1-\varepsilon.\nn \\ 
&& \forall\, \, b \, \mbox{and}\, \varepsilon>0,\,
\exists \,\, \del>0,\,\,\mbox{and}\,\,
n_2 \in \NN \,\mbox{whenever}\, \|\bs\|\le
b,\, \mbox{and}\,
 n>n_2,\label{ll}\\ &&P\left(\sup_{\|\bt-\bs\|\le \del}n^{-1/2}\s |
l_{nt}(\bt)-l_{nt}(\bs) |\le \varepsilon\right)> 1-\varepsilon.\nn
\eenr

Conditions (\r{cnt})-(\r{ll}) are regularity conditions on the weights and perturbations of the two-parameters empirical processes. Conditions (\r{kk})-(\r{ll}) are smoothness conditions on the weights and perturbations. Under stationarity and ergodicity, many of these conditions reduce to much simpler conditions based on existence of the moments. 

The following theorem generalizes (\ref{ti*}) when the weights are functions of $\bb$. 

\begin{theorem}\label{UUbb}
Under the above framework, suppose that conditions (\r{cnt})-(\r{ll}) and Assumption
(A1) hold. Then for every $0<b<\iny$,
\benr
\sup_{x\in \R,\| \bb \|\le b} |\ti
{\cal U}(x, \bb)-{\cal U}^*(x, \bb)|&=&o_{\rm P}(1).\label{uu*}
\eenr
\end{theorem}

\begin{proof}
Clearly, under conditions  (\ref{cnt})-(\ref{g5}), Theorem~\ref{Thm.emp} entails that for each fixed $\bb$,
$$\sup_{x\in \R} |\ti {\cal U}(x, \bb)-{\cal U}^*(x, \bb)| = o_{\rm P}(1).$$
The uniform convergence with respect to $\bb$ over compact sets can be proved as in Mukherjee~(2007, Lemma 3.2) using conditions (\ref{kk}) and (\ref{ll}).
\end{proof}

The following facts are useful in the proofs of various results of this paper. Let $m=1+p+q$ be the total number of parameters and fix $\bb \in \R^m$. Let $\bth_{n\vp}=\bth_{0\vp}+n^{-1/2}\bb$,  
\begin{equation}\label{vntbb}
u_{nt}(\bb)= \frac{v_t^{1/2}(\bth_{n\vp})}{v_t^{1/2}(\bth_{0\vp})}-1, \,\,
v_{nt}(\bb) = \frac{v_t^{1/2}(\bth_{0\vp})}{v_t^{1/2}(\bth_{n\vp})} - 1.
\end{equation} 
Then $\{u_{nt}(\bb)\}$ satisfies (\ref{g3}) since
\begin{equation}\label{vntbb1}
u_{nt}(\bb)=  \frac{v_t(\bth_{n\vp})-v_t(\bth_{0\vp})} {v^{1/2}_t(\bth_{0\vp})\{v_t^{1/2}(\bth_{n\vp})+v_t^{1/2}(\bth_{0\vp})\}}
= \frac{n^{-1/2} \d{\mbf v}_t^\prime (\bth^*) \bb} {v^{1/2}_t(\bth_{0\vp})\{v_t^{1/2}(\bth_{n\vp})+v_t^{1/2}(\bth_{0\vp})\}},  
\end{equation} 
for some $\bth^*=\bth^*(n, t, \bb)$ in the neighbourhood of $\bth_{0\vp}$ for large $n$. The 
$n^{-1/2}$-factor is used later for deriving convergence of some sequence of random vectors. Similarly, for some $\bth^*$,
\begin{equation}\label{vntbb2}
v_{nt}(\bb)=  \frac{v_t(\bth_{0\vp})-v_t(\bth_{n\vp})}{v^{1/2}_t(\bth_{n\vp})\{v_t^{1/2}(\bth_{n\vp})+v_t^{1/2}(\bth_{0\vp})\}}
= \frac{-n^{-1/2} \d{\mbf v}_t^\prime (\bth^*)\bb} {v^{1/2}_t(\bth_{n\vp})\{v_t^{1/2}(\bth_{n\vp})+v_t^{1/2}(\bth_{0\vp})\}}
=n^{-1/2} \xi_{nt},  
\end{equation} 
say. Let $a_{nt}(\bb)=v_t^{1/2}(\bth_{0\vp})/v_t^{1/2}(\bth_{n\vp})=1+v_{nt}(\bb)=1+n^{-1/2} \xi_{nt}$. Then 
$$
\frac{X_t}{v_t^{1/2}(\bth_{n\vp})}=a_{nt}(\bb)\eta_t=\eta_t+ n^{-1/2}\eta_t\xi_{nt}=\eta_t+ n^{-1/2}z_{nt},
$$
where
$$
z_{nt}=\eta_t\xi_{nt}=\eta_t \times \frac{-\d{\mbf v}_t^\prime (\bth^*)\bb} {v^{1/2}_t(\bth_{n\vp})\{v_t^{1/2}(\bth_{n\vp})+v_t^{1/2}(\bth_{0\vp})\}}.
$$
For $\delta >0$ in Assumption (A1) and any $c>0$,
\begin{equation*}
\begin{split}
P\left[ n^{-1/2}  \underset{1 \leq t \leq n}{\max} |z_{nt}| > c \right] \leq \s P\left[ n^{-1/2}  |z_{nt}| > c \right] 
\leq n \frac{{\mathrm E}\left[n^{-1-\delta/2} |\eta_t|^{2+\delta} |\xi_{nt}|^{2+\delta} \right]}{c^{2+\delta}}=o(1)
\end{split}
\end{equation*}
since all moments of $\{|\xi_{nt}|\}$ are finite and $\eta_t$ and $\xi_{nt}$ are independent for all $t$. Therefore
\begin{equation}\label{anvn}
\underset{1 \leq t \leq n}{\max} \left| \frac{X_t}{v_t^{1/2}(\bth_{n\vp})} -\eta_t \right| = o_{\rm P}(1).
\end{equation}
If $\d{\mbf v}_t(\bth_{n\vp})/v_t(\bth_{n\vp})$ appears as the coefficients, we replace it by $\d{\mbf v}_t(\bth_{0\vp})/v_t(\bth_{0\vp})$ and the difference is controlled as follows. Notice that all derivatives below exist with bounded moments and so 
\begin{equation}\label{vvdot.diff}
\frac{\d{\mbf v}_t(\bth_{n\vp})}{v_t(\bth_{n\vp})}-\frac{\d{\mbf v}_t(\bth_{0\vp})}{v_t(\bth_{0\vp})}
=n^{-1/2}\A_t(\bth_{0\vp})\bb+ n^{-1}\A_{tn}^*,
\end{equation}
where $\A_t(\bth_{0\vp})= \ddot{\mbf v}_t(\bth_{0\vp})/v_t(\bth_{0\vp})-\d{\mbf v}_t(\bth_{0\vp})\d{\mbf v}_t^{'}(\bth_{0\vp})/\{v_t(\bth_{0\vp})\}^2$. Only the term $n^{-1/2}\A_t(\bth_{0\vp})\bb$ is of our interest since others are of higher order than $n$. 

Take $l_{nt}(\bb)$ to be equal to the $j$-th coordinate $(1 \le j \le m=1+p+q)$ of
\begin{equation}\label{LN}
{\mbf L}_{nt}(\bb) = \frac{\d{\mbf v}_t(\bth_{n\vp})}{v_t(\bth_{n\vp})} \times \frac{v_t^{1/2}(\bth_{0\vp})}{v_t^{1/2}(\bth_{n\vp})}   
\end{equation}
and $u_{nt}(\bb)$ as in (\ref{vntbb}). We now show that (\ref{l1})-(\ref{ll}) hold with such choice.

For each $t$ with $1\le t\le n$, $\{{\mbf L}_{nt}(\bb), \, u_{nt}(\bb)\}$ are independent of $\eta_t$.
Using a Taylor expansion of  $l_{nt}(\bb)$ at $\bth_{0\vp}$ for each $1 \le t \le n$ and noting the existence of all moments of 
$v_t(\bth_{0\vp})$ and its derivatives of all higher orders, (\ref{l1}) and (\ref{11*}) hold. Existence of all higher moments of $\{l_{nt}(\bb), \, u_{nt}(\bb)\}$ ensure conditions (\r{l2})-(\r{g5*}).  

To verify (\ref{g5}), we use (\ref{vntbb1}) and that for each $t$, $v_t(\cdot)$ is a smooth function with derivatives of all order to conclude that
$$
n^{-1/2} \s {\mbf L}_{nt}(\bb)u_{nt}(\bb)={\mathrm E}[ \d{\mbf v}_1(\bth_{0\vp}) \d{\mbf v}_1^\prime(\bth_{0\vp})/v_1^2(\bth_{0\vp})] (\bb/2) +o_{\rm P}(1)=
\J(\bth_{0\vp}) \bb/2+o_{\rm P}(1).
$$ 
Conditions (\ref{kk}) and (\ref{ll}) can be verified using the mean value theorem.

The following lemmas and their proofs represent the intermediate steps in the proofs of Proposition~\ref{Prop.asy.linear} and Theorem~\ref{thm.asy}.

\textbf{Lemma \ref{lem.Tn21M}, Lemma \ref{lem.Tn1N}, Lemma \ref{lemmaRT.diff} and Lemma \ref{lem.RRhat}.} 

Let 
$$
\T_{n1}(\bth_{n\vp})=n^{-1/2}\s \, \frac{\d{\mbf v}_t(\bth_{n\vp})}{v_t(\bth_{n\vp})}\left\lbrace 1 -\frac{X_t}{v_t^{1/2}(\bth_{n\vp})}\vp[ G(\eta_t)] \right\rbrace, 
$$
$$
\T_{n2}(\bth_{n\vp})=n^{-1/2}\s \, \frac{\d{\mbf v}_t(\bth_{n\vp})}{v_t(\bth_{n\vp})}\left\lbrace    1-\frac{X_t}{v_t^{1/2}(\bth_{n\vp})}\vp\left[ G\left(\frac{X_t}{v_t^{1/2}(\bth_{n\vp})}\right) \right] \right\rbrace 
$$
and note that the difference in the definitions of these two quantities lies only in the argument of $\vp(G(.))$. We show in 
Lemma \ref{lem.Tn21M} below that $\int_0^1 [\ti{\cal V}(u,\bb)-{\cal V}^*(u,\bb)] d\vp(u)=\T_{n1}(\bth_{n\vp})-\T_{n2}(\bth_{n\vp})$. Using results on empirical processes in Theorem \ref{UUbb}, $\int_0^1 [\ti{\cal V}(u,\bb)-{\cal V}^*(u,\bb)] d\vp(u)$ is linear in $\bb$. Consequently, we obtain the following uniform approximations of $\T_{n1}(\bth_{n\vp})-\T_{n2}(\bth_{n\vp})$ over $||\bb|| \le c$ where $c>0$. 

\begin{lemma}\label{lem.Tn21M}
Let Assumptions (A1)-(A3) hold. Then, as $n \rightarrow \infty$,
\begin{equation}\label{Tn21M}
\T_{n2}(\bth_{n\vp}) - \T_{n1}(\bth_{n\vp}) = \M(\bth_{0\vp}) \bb + u_{\rm P}(1),
\end{equation}
where $\M(\bth_{0\vp}) = \J(\bth_{0\vp}) \rho(\vp)/2$.
\end{lemma}

\begin{proof}
To use Theorem \ref{UUbb} in the proof, let $\bb=n^{1/2}(\bth_{n\vp}-\bth_{0\vp})$ and $x=G^{-1}(u)$ for some $0<u<1$. For simplicity, we use the notation $\ti{\cal V}(u,\bb)$ to denote $\ti{\cal V}(G^{-1}(u),\bb)$ which is defined in the probabilistic framework above. Accordingly 
$$\ti{\cal V}(u,\bb):= n^{-1/2}\s \, \frac{\d{\mbf v}_t(\bth_{n\vp})}{v_t(\bth_{n\vp})}  \frac{X_t}{v_t^{1/2}(\bth_{n\vp})} I\left[ \eta_t< G^{-1}(u) 
\frac{v_t^{1/2}(\bth_{n\vp})}{v_t^{1/2}(\bth_{0\vp})}\right] $$
and
$${\cal V}^*(u,\bb)= n^{-1/2}\s \, \frac{\d{\mbf v}_t(\bth_{n\vp})}{v_t(\bth_{n\vp})}  \frac{X_t}{v_t^{1/2}(\bth_{n\vp})} I\left(\eta_t < G^{-1}(u) \right).
$$
With the choice based on (\ref{LN}) and (\ref{vntbb}) and using
$$
\frac{v_t^{1/2}(\bth_{0\vp})}{v_t^{1/2}(\bth_{n\vp})}\eta_t=\frac{X_t}{v_t^{1/2}(\bth_{n\vp})},
$$
\begin{eqnarray*}
\ti{\cal V}(u,\bb) &=& n^{-1/2}\s \, \frac{\d{\mbf v}_t(\bth_{n\vp})}{v_t(\bth_{n\vp})}  \frac{X_t}{v_t^{1/2}(\bth_{n\vp})} I\left[\frac{X_t}{v_t^{1/2}(\bth_{n\vp})}< G^{-1}(u) \right] \\
&=& n^{-1/2}\s \, \frac{\d{\mbf v}_t(\bth_{n\vp})}{v_t(\bth_{n\vp})}  \frac{X_t}{v_t^{1/2}(\bth_{n\vp})} I\left[ G \left(\frac{X_t}{v_t^{1/2}(\bth_{n\vp})}\right)< u \right].
\end{eqnarray*}
Similarly,
$$
{\cal V}^*(u,\bb)= n^{-1/2}\s \, \frac{\d{\mbf v}_t(\bth_{n\vp})}{v_t(\bth_{n\vp})}  \frac{X_t}{v_t^{1/2}(\bth_{n\vp})} I\left(G(\eta_t) < u \right).
$$
Since 
$$
\int_0^1 I\left\lbrace G \left(\frac{X_t}{v_t^{1/2}(\bth_{n\vp})}\right)< u \right\rbrace d\vp(u)=\vp(1)-\vp \left[ G\left(\frac{X_t}{v_t^{1/2}(\bth_{n\vp})}\right) \right],
$$ 
we get
$$
\int_0^1 \ti{\cal V}(u,\bb) d\vp(u) =n^{-1/2}\s \, \frac{\d{\mbf v}_t(\bth_{n\vp})}{v_t(\bth_{n\vp})}  \frac{X_t}{v_t^{1/2}(\bth_{n\vp})} \left\lbrace \vp(1)- \vp \left[ G\left(\frac{X_t}{v_t^{1/2}(\bth_{n\vp})}\right) \right] \right\rbrace
$$
and
$$
\int_0^1 {\cal V}^*(u,\bb) d\vp(u) =n^{-1/2}\s \, \frac{\d{\mbf v}_t(\bth_{n\vp})}{v_t(\bth_{n\vp})}   \frac{X_t}{v_t^{1/2}(\bth_{n\vp})} \Big(\vp(1)-\vp(G(\eta_t)\Big).
$$
Cancelling $\vp(1)$, $\int_0^1 [\ti{\cal V}(u,\bb)-{\cal V}^*(u,\bb)] d\vp(u)$ equals 
\begin{align*}
&\quad n^{-1/2}\s \, \frac{\d{\mbf v}_t(\bth_{n\vp})}{v_t(\bth_{n\vp})} \frac{X_t}{v_t^{1/2}(\bth_{n\vp})} \left\lbrace -\vp \left[ G\left(\frac{X_t}{v_t^{1/2}(\bth_{n\vp})}\right) \right]
+\vp(G(\eta_t)\Big) \right\rbrace \\
&= \T_{n2}(\bth_{n\vp}) - \T_{n1}(\bth_{n\vp}).
\end{align*}
Using (\ref{JJ*}) and (\ref{vntbb1}) with
$$
\ti{\cal J}(u,\bb):= n^{-1/2}\s \, \frac{\d{\mbf v}_t(\bth_{n\vp})}{v_t(\bth_{n\vp})}  \frac{v_t^{1/2}(\bth_{0\vp})}{v_t^{1/2}(\bth_{n\vp})}
\mu \left[ G^{-1}(u) 
\frac{v_t^{1/2}(\bth_{n\vp})}{v_t^{1/2}(\bth_{0\vp})}\right]
$$
$$
{\cal J}^*(u,\bb) := n^{-1/2}\s \, \frac{\d{\mbf v}_t(\bth_{n\vp})}{v_t(\bth_{n\vp})}  \frac{v_t^{1/2}(\bth_{0\vp})}{v_t^{1/2}(\bth_{n\vp})}
\mu \left( G^{-1}(u) \right),
$$
we have
$$
\sup_{u\in (0, 1)}  
|\ti{\cal J}(u,\bb)-{\cal J}^*(u,\bb)-
[G^{-1}(u)]^2 g(G^{-1}(u)) n^{-1/2}\s \frac{\d{\mbf v}_t(\bth_{n\vp})}{v_t(\bth_{n\vp})}  \frac{v_t^{1/2}(\bth_{0\vp})}{v_t^{1/2}(\bth_{n\vp})} u_{nt}(\bb)|
=u_{\rm P}(1).
$$
Also, 
$$|n^{-1/2}\s \frac{\d{\mbf v}_t(\bth_{n\vp})}{v_t(\bth_{n\vp})}  \frac{v_t^{1/2}(\bth_{0\vp})}{v_t^{1/2}(\bth_{n\vp})} u_{nt}(\bb)-\J(\bth_{0\vp})|
=u_{\rm P}(1).$$
Hence,
\begin{align}
\int_0^1 [\ti{\cal J}(u,\bb)-{\cal J}^*(u,\bb)] d\vp(u)&= \int_0^1 [G^{-1}(u)]^2 g(G^{-1}(u)) d\vp(u) \J \bb/2 +u_{\rm P}(1)  \nonumber \\
&= \M(\bth_{0\vp}) \bb +u_{\rm P}(1) \label{JJdiff}
\end{align}
by recalling that $\M(\bth_{0\vp}) = \J(\bth_{0\vp}) \rho(\vp)/2$. Finally, (\ref{Tn21M}) follows from Theorem \ref{UUbb}.
\end{proof}
The following lemma states that the difference between $\T_{n1}(\bth_{n\vp})$ and $\N_n(\bth_{0\vp})$ is asymptotically linear in $\bb$.
\begin{lemma}\label{lem.Tn1N}
Let Assumptions (A1)-(A3) hold. Then, as $n \rightarrow \infty$,
\begin{equation}\label{T1n}
\T_{n1}(\bth_{n\vp}) -\N_n(\bth_{0\vp}) = \J(\bth_{0\vp}) \bb/2+u_{\rm P}(1), 
\end{equation}
where 
\begin{equation}\label{Nn.Norm}
\N_n(\bth_{0\vp}) \rightarrow {\mathcal N} ( \0, \J(\bth_{0\vp}) \sigma^2(\vp)),
\end{equation}
with $\sigma^2(\vp)={\rm Var}\{\eta_1 \vp [G(\eta_1)] \}$.
\end{lemma}

\begin{proof}
The difference between $\T_{n1}(\bth_{n\vp})$ and $\N_n(\bth_{0\vp})$ lies in comparing $X_t/v_t^{1/2}(\bth_{n\vp})=\eta_t+ n^{-1/2}z_{nt}$ and $\eta_t$ and involves  smooth function of $\bb$. So the proof follows easily with the details below. Notice that
\begin{eqnarray*}
\T_{n1}(\bth_{n\vp})-\N_n(\bth_{0\vp}) &=& n^{-1/2}\s \, \left[\frac{\d{\mbf v}_t(\bth_{n\vp})}{v_t(\bth_{n\vp})}-\frac{\d{\mbf v}_t(\bth_{0\vp})}{v_t(\bth_{0\vp})}\right] \left\lbrace 1- a_{nt}(\bb)\eta_t \vp[G(\eta_t)] \right\rbrace \\
&- & n^{-1/2}\s \, \frac{\d{\mbf v}_t(\bth_{0\vp})}{v_t(\bth_{0\vp})} v_{nt}(\bb) \eta_t \vp[G(\eta_t)] =  {\mbf F}_{n1}- {\mbf F}_{n2}.
\end{eqnarray*}
Using (\ref{vvdot.diff}),
\begin{eqnarray*}
{\mbf F}_{n1}&=& n^{-1}\s \, \A_t(\bth_{0\vp})\bb \left\lbrace 1- a_{nt}(\bb)\eta_t \vp[G(\eta_t)] \right\rbrace +u_{\rm P}(1) \\
&=& n^{-1}\s \, \A_t(\bth_{0\vp})\bb \left\lbrace 1- \eta_t \vp[G(\eta_t)] \right\rbrace
- n^{-1}\s \, \A_t(\bth_{0\vp})\bb   v_{nt}(\bb)\eta_t \vp[G(\eta_t)]  +u_{\rm P}(1). 
\end{eqnarray*}
Using the LLN, the first term in the above decomposition of ${\mbf F}_{n1}$ is $u_{\rm P}(1)$ since 
$${\mathrm E}\left\lbrace \A_t(\bth_{0\vp})\bb \left\lbrace 1- \eta_t \vp[G(\eta_t)] \right\rbrace \right\rbrace ={\mathrm E}[\A_t(\bth_{0\vp})\bb] \, {\mathrm E}\left\lbrace 1- \eta_t \vp[G(\eta_t)] \right\rbrace =\0.$$
For the second term, using (\ref{vntbb2}) we have $n^{-1/2}$ factor of $v_{nt}(\bb)$ and consequently it is $u_{\rm P}(1)$.

For ${\mbf F}_{n2}$, we approximate $v_{nt}(\bb)$ by $-n^{-1/2}\d{\mbf v}_t^\prime(\bth_{0\vp}) \bb/\{2 v_t (\bth_{0\vp}) \}$ and use ${\rm E}\{\eta_t \vp[G(\eta_t)]\}=1$ to obtain 
${\mbf F}_{n2}=-\J(\bth_{0\vp}) \bb/2+u_{\rm P}(1)$. Hence (\ref{T1n}) is proved. 

Using the independence of $v_t$ and $\eta_t$ for each $t$, $\N_n(\bth_{0\vp})$ is a sum of the vectors of martingale differences and so
(\ref{Nn.Norm}) follows from the martingale CLT.
\end{proof}

Now consider the rank-based counterpart of $\T_{n2}(\bth_{n\vp})$
$${\mbf R}_n(\bth_{n\vp})= n^{-1/2}  \s \, \frac{\d{\mbf v}_t(\bth_{n\vp})}{v_t(\bth_{n\vp})}  \left\lbrace 1- \frac{X_t}{v_t^{1/2}(\bth_{n\vp})}\vp \left(\frac{R_{nt}(\bth_{n\vp})}{n+1} \right) \right\rbrace.$$
The following lemma provides the difference between $\T_{n2}(\bth_{n\vp})$ and ${\mbf R}_n(\bth_{n\vp})$. It shows that the effect of replacing observations in $\T_{2n}(\bth_{n\vp})$ by ranks is asymptotically a r.v. with mean zero. 
\begin{lemma}\label{lemmaRT.diff}
Let Assumptions (A1)-(A3) hold. Then, as $n \rightarrow \infty$,
\begin{equation}\label{rank2}
{\mbf R}_n(\bth_{n\vp})-\T_{n2}(\bth_{n\vp}) = \Q_n(\bth_{0\vp}) + u_{\rm P}(1).
\end{equation}
Also, $\Q_n(\bth_{0\vp})$ converges in distribution to ${\mathrm E}(\d{\mbf v}_1(\bth_{0\vp})/v_1(\bth_{0\vp}) ) Z$, where $Z$ has mean zero and variance $\gamma(\vp)$.
\end{lemma}

\begin{proof}
Consider the following decomposition
\begin{eqnarray*}
\lefteqn{{\mbf R}_n(\bth_{n\vp})-\T_{n2}(\bth_{n\vp})}\\
&=& n^{-1/2}\s \, \frac{\d{\mbf v}_t(\bth_{n\vp})}{v_t(\bth_{n\vp})}\frac{X_t}{v_t^{1/2}(\bth_{n\vp})} \left\lbrace \vp \left[ G\left(\frac{X_t}{v_t^{1/2}(\bth_{n\vp})}\right) \right]
-\vp \left[ \frac{R_{nt}(\bth_{n\vp})}{n+1} \right] \right\rbrace \\
& = & n^{-1/2}\s \, \left[\frac{\d{\mbf v}_t(\bth_{n\vp})}{v_t(\bth_{n\vp})}-\frac{\d{\mbf v}_t(\bth_{0\vp})}{v_t(\bth_{0\vp})}\right]  \eta_t \left\lbrace \vp \left[ G\left(\frac{X_t}{v_t^{1/2}(\bth_{n\vp})}\right) \right]
- \vp \left[ \frac{R_{nt}(\bth_{n\vp})}{n+1} \right]  \right\rbrace \\
&+& n^{-1/2}\s \, \left[\frac{\d{\mbf v}_t(\bth_{n\vp})}{v_t(\bth_{n\vp})}-\frac{\d{\mbf v}_t(\bth_{0\vp})}{v_t(\bth_{0\vp})}\right]   v_{nt}(\bb)  \eta_t \left\lbrace \vp \left[ G\left(\frac{X_t}{v_t^{1/2}(\bth_{n\vp})}\right) \right]
-\vp \left[ \frac{R_{nt}(\bth_{n\vp})}{n+1} \right]  \right\rbrace \\
&+& n^{-1/2}\s \, \frac{\d{\mbf v}_t(\bth_{0\vp})}{v_t(\bth_{0\vp})} \eta_t \left\lbrace \vp \left[ G\left(\frac{X_t}{v_t^{1/2}(\bth_{n\vp})}\right) \right]
-\vp \left[ \frac{R_{nt}(\bth_{n\vp})}{n+1} \right] \right\rbrace \\
&+& n^{-1/2}\s \, \frac{\d{\mbf v}_t(\bth_{0\vp})}{v_t(\bth_{0\vp})}  v_{nt}(\bb) \eta_t \left\lbrace \vp \left[ G\left(\frac{X_t}{v_t^{1/2}(\bth_{n\vp})}\right) \right] 
-\vp \left[ \frac{R_{nt}(\bth_{n\vp})}{n+1} \right] \right\rbrace \\
&=& \D_{n1} + \D_{n2} + \D_{n3} + \D_{n4}.
\end{eqnarray*}
Using the $n^{-1/2}$-factor in (\ref{vvdot.diff}) and (\ref{vntbb2}), $\D_{n1}$, $\D_{n2}$ and $\D_{n4}$ are $u_{\rm P}(1)$. We next prove that $\D_{n3}=\Q_n(\bth_{0\vp}) +u_{\rm P}(1)$ in detail. Recall that $\tilde{G}_n(x)$ is the empirical distribution function of $\{\eta_t\}$. Let $G_n(x), x\in \R$ be the empirical distribution function of $\{X_t/v_t^{1/2}(\bth_{n\vp})\}$. Then 
\begin{eqnarray*}
\D_{n3}&= & n^{-1/2}\s \, \frac{\d{\mbf v}_t(\bth_{0\vp})}{v_t(\bth_{0\vp})} \eta_t \left\lbrace \vp \left[ G\left(\frac{X_t}{v_t^{1/2}(\bth_{n\vp})}\right) \right]
- \vp \left[ G_n \left(\frac{X_t}{v_t^{1/2}(\bth_{n\vp})}\right) \right] \right\rbrace \\
&=& n^{-1/2}\s \, \frac{\d{\mbf v}_t(\bth_{0\vp})}{v_t(\bth_{0\vp})} \eta_t \left\lbrace \vp \left[ G\left(\frac{X_t}{v_t^{1/2}(\bth_{n\vp})}\right) \right]
- \vp \left[ G(\eta_t) \right] \right\rbrace \\
& - &   n^{-1/2}\s \, \frac{\d{\mbf v}_t(\bth_{0\vp})}{v_t(\bth_{0\vp})} \eta_t \left\lbrace \vp \left[ G_n\left(\frac{X_t}{v_t^{1/2}(\bth_{n\vp})}\right) \right]
- \vp \left[ \tilde{G}_n(\eta_t) \right] \right\rbrace \\
& + &   n^{-1/2}\s \, \frac{\d{\mbf v}_t(\bth_{0\vp})}{v_t(\bth_{0\vp})} \eta_t \left\lbrace \vp \left[ G(\eta_t) \right]
- \vp \left[ \tilde{G}_n(\eta_t) \right] \right\rbrace \\
&=& \D_{n1}^* - \D_{n2}^* + \D_{n3}^*.
\end{eqnarray*}
Since $\D_{n1}^*$ is the weighted sum of the difference of a c.d.f. evaluated at two different r.v.'s and integrated wrt $\vp$, using the same technique for proving (\ref{JJdiff}), $\D_{n1}^*=\M(\bth_{0\vp}) \bb +u_{\rm P}(1)$.

Write $\w_t = \d{\mbf v}_t(\bth_{0\vp})/v_t(\bth_{0\vp})$. Since $\D_{n2}^*$ is the weighted sum of the difference of two different c.d.f.'s evaluated at two different r.v.'s and integrated wrt $\vp$,
\begin{eqnarray*}
\D_{n2}^* &=& \int_{0}^{1} n^{-1/2} \s \w_t \eta_t I \left[ G_n\left(\frac{X_t}{v_t^{1/2}(\bth_{n\vp})} \right) < u \right] - I \left[ \tilde{G}_n(\eta_t) < u \right] d \varphi(u)  \\
&= & \int_{0}^{1} n^{-1/2} \s \w_t \eta_t  I \left[ \frac{X_t}{v_t^{1/2}(\bth_{n\vp})}  < G_n^{-1}(u) \right] - I \left[ \eta_t < \tilde{G}_n^{-1}(u)  \right]  d \varphi(u)  \\
&= & \int_{0}^{1} n^{-1/2} \s \w_t \eta_t I \left[ \eta_t  < G_n^{-1}(u) \frac{1}{1 + v_{nt}(\bb)}  \right] - I \left[ \eta_t < \tilde{G}_n^{-1}(u)  \right] d \varphi(u). 
\end{eqnarray*}
Using (\ref{anvn}), $\sup\left\lbrace \left| G_n^{-1}(u) - \tilde{G}_n^{-1}(u)\right|; u \in (0, 1)\right\rbrace = u_{\rm P}(1)$. Hence, by Theorem \ref{UUbb},
\begin{eqnarray*}
\D_{n2}^*&=& \int_{0}^{1} n^{-1/2} \s \w_t  \left[ \mu \left( \tilde{G}_n^{-1}\left(u \right)\frac{1}{1+v_{nt}(\bb)} \right) - \mu\left( \tilde{G}_n^{-1}\left(u \right) \right) \right] d \varphi(u) + u_{\rm P}(1) \\
&=& \int_{0}^{1} n^{-1/2} \s \w_t \d{\mu} \left( \tilde{G}_n^{-1}\left(u \right) \right) \frac{-v_{nt}(\bb)}{1+v_{nt}(\bb)} \tilde{G}_n^{-1}\left(u \right) d \varphi(u) + u_{\rm P}(1) \\
&=& \int_{0}^{1} n^{-1/2} \s \w_t \left(\tilde{G}_n^{-1}\left(u \right) \right)^2 g\left(\tilde{G}_n^{-1}\left(u \right) \right)  \frac{-v_{nt}(\bb)}{1+v_{nt}(\bb)} d \varphi(u) + u_{\rm P}(1)  \\
&=&   \int_{0}^{1} n^{-1} \s \frac{\d{\mbf v}_t(\bth_{0\vp}) \d{\mbf v}_t^\prime (\bth_{0\vp})}{2 v_t^2(\bth_{0\vp})} \left(G^{-1}(u)\right)^2 g\left(G^{-1}(u) \right)  d \varphi(u) \bb + 
u_{\rm P}(1) \\
&=& \M(\bth_{0\vp}) \bb + u_{\rm P}(1).
\end{eqnarray*}
Finally consider $\D_{n3}^*$ written as
\begin{eqnarray*}
\D_{n3}^*&=& \int_{0}^{1} n^{-1/2} \s \w_t \eta_t \left\lbrace I \left[ \eta_t \leq G^{-1}(u) \right] - I \left[ \eta_t \leq \tilde{G}_n^{-1}(u) \right] \right\rbrace d \varphi(u)  \\
&=&  \int_{0}^{1} n^{-1/2} \s \w_t \eta_t \left\lbrace I \left[ \eta_t \leq G^{-1}(u) \right] - I \left[ \eta_t \leq G^{-1}(G(\tilde{G}_n^{-1}(u))) \right] \right\rbrace  d \varphi(u)  \\
&=& \int_{0}^{1} \left[ \M_n(u) - \M_n(G(\tilde{G}_n^{-1}(u))) \right]  d \varphi(u) \\
&+& \int_{0}^{1}  n^{-1/2} \s \w_t \left[ \mu(G^{-1}(u)) -  \mu(\tilde{G}_n^{-1}(u))\right]   d \varphi(u),
\end{eqnarray*}
where $\M_n(u) := n^{-1/2} \s \w_t \left\lbrace \eta_t  I \left[ \eta_t \leq G^{-1}(u) \right]  - \mu(G^{-1}(u)) \right\rbrace$.
We show that
\begin{equation}\label{MM1}
\left| \int_{0}^{1} \left[ \M_n(u) - \M_n(G(\tilde{G}_n^{-1}(u))) \right]  d \varphi(u) \right| = o_{\rm P}(1),
\end{equation}
\begin{equation}\label{muGGn1}
\Q_n(\bth_{0\vp}) = \int_{0}^{1}  n^{-1/2} \s \w_t \left[ \mu(G^{-1}(u)) -  \mu(\tilde{G}_n^{-1}(u))\right]   d \varphi(u) \rightarrow {\mathrm E}(\d{\mbf v}_1/v_1) Z.
\end{equation}
For (\ref{MM1}), note that $\{\M_n(.)\}$ converges weakly to a Brownian Bridge on $(0, 1)$ since for each fixed $u$, 
$\M_n(u)$ converges to a normal distribution using the martingale CLT and it is tight using the bound on the moment of the difference process in Billingsley~(1968, Theorem~12.3).

Since $\sup\{|u - G(\tilde{G}_n^{-1}(u))|; u \in (0, 1)\} = \sup\{|G(x) - \tilde{G}_n(x)|; x \in \R \} = o_{\rm P}(1)$, by the Arzela-Ascoli theorem,
\begin{equation*}
\sup \left\lbrace \left|  \M_n(u) - \M_n(G(\tilde{G}_n^{-1}(u)))  \right|;u \in (0, 1)\right\rbrace =o_{\rm P}(1),
\end{equation*}
and consequently, (\ref{MM1}) is proved. For (\ref{muGGn1}), we use the Bahadur representation; see Bahadur~(1966)  and Ghosh~(1971) for details. Since $\dot{g}(x)$ is bounded and $g$ is positive
on $\R$, 
\begin{equation*}
n^{1/2} \left(G^{-1}(u)) -  \tilde{G}_n^{-1}(u)\right) - n^{-1/2} \sum_{i=1}^n \frac{I\{\eta_i \leq G^{-1}(u) \} - u}{g\left(G^{-1}(u) \right)}  = o(1) \quad \text{a.s.}.
\end{equation*}
Applying the mean value theorem,
\begin{align}
n^{1/2} \left[\mu(G^{-1}(u)) -  \mu(\tilde{G}_n^{-1}(u))\right] - \d{\mu}(G^{-1}(u))  n^{-1/2} \sum_{i=1}^n \frac{I\{\eta_i \leq G^{-1}(u) \} - u}{g\left(G^{-1}(u) \right)} = o(1) \quad \text{a.s.}.\label{Qn1}
\end{align}
Using $\d{\mu}(x)=xg(x)$, 
\begin{align*}
\Q_n(\bth_{0\vp}) &= n^{-1} \s \w_t \int_{0}^{1} \left[\d{\mu}(G^{-1}(u))  n^{-1/2} \sum_{i=1}^n \frac{I\{\eta_i \leq G^{-1}(u) \} - u}{g\left(G^{-1}(u) \right)} \right]  d \varphi(u)+o_{\rm P}(1)\\
&=n^{-1} \s \w_t \int_{0}^{1} \left[ G^{-1}(u) n^{1/2} \left( \tilde{G}_n(G^{-1}(u)) - u\right) \right]  d \varphi(u) + o_{\rm P}(1).
\end{align*}
Since $n^{-1} \s \w_t \rightarrow {\mathrm E}(\d{\mbf v}_1(\bth_{0\vp})/v_1(\bth_{0\vp}))$, and using van der Vaart~(1998, Theorem~19.3),
$$n^{1/2} \left( \tilde{G}_n(G^{-1}(u)) - u\right) \rightarrow B(u),$$
we obtain (\ref{muGGn1}) with the r.v. $Z$ having mean zero. The variance of $Z$ is given by
\begin{eqnarray*}
{\mathrm E}(Z^2)&=& {\mathrm E}\left[ \int_0^1 \int_0^1 G^{-1}(u) G^{-1}(v) B(u) B(v) d \vp(u) d \vp(v)  \right] \\
&=& \int_0^1 \int_0^1 G^{-1}(u) G^{-1}(v) {\mathrm E}[B(u) B(v)] d \vp(u) d \vp(v)  \\
&=& \int_0^1 \int_0^1 G^{-1}(u) G^{-1}(v) \left[  \min\{u, v\} - uv \right] d \vp(u) d \vp(v)  \\
&= & \gamma(\vp).
\end{eqnarray*}
\end{proof}

Now recall the rank-based central sequence
$$\h{{\mbf R}}_n(\bth_{n\vp})=n^{-1/2}\s \, \frac{\d{\h{\mbf v}}_t(\bth_{n\vp})}{\h{v}_t(\bth_{n\vp})}  \left\lbrace 1- \frac{X_t}{\h{v}_t^{1/2}(\bth_{n\vp})}\vp \left(\frac{\h{R}_{nt}(\bth_{n\vp})}{n+1} \right) \right\rbrace,$$
which is an approximation to ${\mbf R}_n(\bth_{n\vp})$. We have the following lemma dealing with the difference between 
${\mbf R}_n(\bth_{n\vp})$ and $\h{{\mbf R}}_n(\bth_{n\vp})$.
\begin{lemma}\label{lem.RRhat}
Let Assumptions (A1)-(A3) hold. Then, as $n \rightarrow \infty$,
\begin{equation}\label{rank3}
{\mbf R}_n(\bth_{n\vp}) - \h{{\mbf R}}_n(\bth_{n\vp}) = u_{\rm P}(1).
\end{equation}
\end{lemma}

\begin{proof}
Note that ${\mbf R}_n(\bth_{n\vp}) - \h{{\mbf R}}_n(\bth_{n\vp})$ equals
\begin{align}
&\quad n^{-1/2}\s \left[ \frac{\d{\mbf v}_t(\bth_{n\vp})}{v_t(\bth_{n\vp})} -  \frac{\d{\h{\mbf v}}_t(\bth_{n\vp})}{\h{v}_t(\bth_{n\vp})} \right] \label{RR1}\\
& + n^{-1/2}\s  \left[\frac{\d{\h{\mbf v}}_t(\bth_{n\vp})}{\h{v}_t(\bth_{n\vp})} \frac{X_t}{\h{v}_t^{1/2}(\bth_{n\vp})} -  \frac{\d{\mbf v}_t(\bth_{n\vp})}{v_t(\bth_{n\vp})} \frac{X_t}{v_t^{1/2}(\bth_{n\vp})}\right] \vp \left(\frac{\h{R}_{nt}(\bth_{n\vp})}{n+1} \right) \label{RR2}\\
&- n^{-1/2}\s  \frac{\d{\mbf v}_t(\bth_{n\vp})}{v_t(\bth_{n\vp})} \frac{X_t}{v_t^{1/2}(\bth_{n\vp})} \left[ \vp \left(\frac{R_{nt}(\bth_{n\vp})}{n+1} \right) - \vp \left(\frac{\h{R}_{nt}(\bth_{n\vp})}{n+1} \right) \right]. \label{RR3}
\end{align}
Due to (\ref{fact21}), (\ref{fact22}) and $\h{v}_t (\bth) \geq c_0(\bth) > 0$, we have
\begin{align}
\left|\frac{\d{\h{\mbf v}}_t(\bth_{n\vp})}{\h{v}_t(\bth_{n\vp})} -\frac{\d{\mbf v}_t(\bth_{n\vp})}{v_t(\bth_{n\vp})}\right|
&= \left|\frac{\d{\h{\mbf v}}_t(\bth_{n\vp}) - \d{\mbf v}_t(\bth_{n\vp})}{\h{v}_t(\bth_{n\vp})} + \d{\mbf v}_t(\bth_{n\vp}) \frac{v_t(\bth_{n\vp}) - \h{v}_t(\bth_{n\vp})}{\h{v}_t(\bth_{n\vp}) v_t(\bth_{n\vp})} \right| \nonumber \\
&\leq \frac{|\d{\h{\mbf v}}_t(\bth_{n\vp}) - \d{\mbf v}_t(\bth_{n\vp})|}{\h{v}_t(\bth_{n\vp})} + \frac{|v_t(\bth_{n\vp}) - \h{v}_t(\bth_{n\vp})|}{\h{v}_t(\bth_{n\vp})}\frac{|\d{\mbf v}_t(\bth_{n\vp})|}{v_t(\bth_{n\vp})} \nonumber  \\
&\leq C\rho^t\left[Z_1 + Z_0 \frac{|\d{\mbf v}_t(\bth_{n\vp})|}{v_t(\bth_{n\vp})}\right]. \label{eq:vtHat}
\end{align}
Hence, in view of (\ref{fact1}) and (\ref{fact3}), for every $0< b <\infty$,
$$\underset{||\bb|| < b}{\sup} \, \s \, \left| \frac{\d{\h{\mbf v}}_t(\bth_{n\vp})}{\h{v}_t(\bth_{n\vp})} -\frac{\d{\mbf v}_t(\bth_{n\vp})}{v_t(\bth_{n\vp})} \right| =  O_{\rm P}(1),$$
which implies that (\ref{RR1}) is $u_{\rm P}(1)$. Since $\vp$ is bounded, (\ref{RR2}) is $u_{\rm P}(1)$. For (\ref{RR3}), since there is a $n^{-1/2}$ factor from
$$\frac{\d{\mbf v}_t(\bth_{n\vp})}{v_t(\bth_{n\vp})} \frac{X_t}{v_t^{1/2}(\bth_{n\vp})} - \frac{\d{\mbf v}_t(\bth_{0\vp})}{v_t(\bth_{0\vp})} \eta_t \, ,$$
it suffices to prove that 
$$\K_n := n^{-1/2}\s  \frac{\d{\mbf v}_t(\bth_{0\vp})}{v_t(\bth_{0\vp})} \eta_t \left[ \vp \left(\frac{R_{nt}(\bth_{n\vp})}{n+1} \right) - \vp \left(\frac{\h{R}_{nt}(\bth_{n\vp})}{n+1} \right) \right]=u_{\rm P}(1).$$
Let $\floor*{x}$ denotes the greatest integer less than or equal to $x$. We split the sum in $\K_n$ into two parts: in the first part, $t$ runs till $\floor*{n^k} - 1$ where $0<k<1/2$. We show that this part is $u_p(1)$ by noting that its expectation is of the form 
$n^k n^{-1/2}=o(1)$ multiplied by expectation of $\d{\mbf v}_t(\bth_{0\vp})/v_t(\bth_{0\vp}) \eta_t$ and a bounded quantity because $\vp$ is bounded. The number of summands in the second term is $n-n^k$ which is large but there we bound expectation of the sum of by a quantity of the form $n \rho^{\floor*{n^k}}$ with $0<k<1/2$ and $0<\rho<1$ and this is $o(1)$.
Accordingly  
\begin{align}
&\K_n=\quad n^{-1/2} \sum_{t=1}^{\floor*{n^k} - 1} \frac{\d{\mbf v}_t(\bth_{0\vp})}{v_t(\bth_{0\vp})} \eta_t \left[ \vp \left(\frac{R_{nt}(\bth_{n\vp})}{n+1} \right) - \vp \left(\frac{\h{R}_{nt}(\bth_{n\vp})}{n+1} \right) \right] \nonumber \\
&+n^{-1/2} \sum_{t=\floor*{n^k}}^{n} \frac{\d{\mbf v}_t(\bth_{0\vp})}{v_t(\bth_{0\vp})} \eta_t\left[ \vp \left(\frac{R_{nt}(\bth_{n\vp})}{n+1} \right) - \vp \left(\frac{\h{R}_{nt}(\bth_{n\vp})}{n+1} \right) \right]. \label{Kfloor}
\end{align}
To show (\ref{Kfloor}) is $u_{\rm P}(1)$, we prove that for every $0<b <\infty$,
\begin{equation}
\underset{\substack{\floor*{n^k} \leq t \leq n \\ ||\bb|| < b}}{\sup} \left| \varphi\left( \frac{R_{nt}(\bth_{n\vp})}{n+1} \right) - \varphi\left( \frac{\h{R}_{nt}(\bth_{n\vp})}{n+1}\right) \right| = O_{\rm P}(n^{k-1}).\label{vpRRhat.unif}
\end{equation}
Since sequences $\{R_{nt}(\bth_{n\vp}) \}$ and $\{\hat{R}_{nt}(\bth_{n\vp}) \}$ are permutations of $\{1, ..., n\}$, with the probability tending to one as $n \rightarrow \infty$, both $\{R_{nt}(\bth_{n\vp}) \}$ and 
$\{\hat{R}_{nt}(\bth_{n\vp}) \}$ are at points of continuity of $\varphi$ that has a finite number of the points of discontinuity. Therefore, to prove (\ref{vpRRhat.unif}), it suffices to prove
\begin{equation}
\underset{\substack{\floor*{n^k} \leq t \leq n \\ ||\bb|| < b}}{\sup} \left|\frac{R_{nt}(\bth_{n\vp})}{n+1} - \frac{\h{R}_{nt}(\bth_{n\vp})}{n+1} \right| = O_{\rm P}(n^{k-1}).\label{RRhat.unif}
\end{equation}
For $\floor*{n^k} \leq t \leq n$, we decompose ranks as 
\begin{align}
&\quad \frac{R_{nt}(\bth_{n\vp})}{n+1} - \frac{\h{R}_{nt}(\bth_{n\vp})}{n+1}  \nonumber \\
&= \frac{1}{n+1} \sum_{i=1}^{\floor*{n^k} - 1}  \left\lbrace I\left[ \frac{X_i}{v_i^{1/2}(\bth_{n\vp})} < \frac{X_t}{v_t^{1/2}(\bth_{n\vp})}  \right] - I\left[ \frac{X_i}{\h{v}_i^{1/2}(\bth_{n\vp})} < \frac{X_t}{\h{v}_t^{1/2}(\bth_{n\vp})}  \right] \right\rbrace \nonumber \\
&\quad + \frac{1}{n+1} \sum_{i=\floor*{n^k}}^{n}  \left\lbrace I\left[ \frac{X_i}{v_i^{1/2}(\bth_{n\vp})} < \frac{X_t}{v_t^{1/2}(\bth_{n\vp})}  \right] - I\left[ \frac{X_i}{\h{v}_i^{1/2}(\bth_{n\vp})} < \frac{X_t}{\h{v}_t^{1/2}(\bth_{n\vp})}  \right] \right\rbrace, \label{qit}
\end{align}
where the first sum is $O_{\rm P}(n^{k-1})$. For the second sum, writing 
\begin{equation*}
I\left[ \frac{X_i}{\h{v}_i^{1/2}(\bth_{n\vp})} < \frac{X_t}{\h{v}_t^{1/2}(\bth_{n\vp})}  \right] = I \left[\frac{X_i}{v_i^{1/2}(\bth_{n\vp})} \frac{v_i^{1/2}(\bth_{n\vp}) \h{v}_t^{1/2}(\bth_{n\vp})}{\h{v}_i^{1/2}(\bth_{n\vp}) v_t^{1/2}(\bth_{n\vp})}  <   \frac{X_t}{v_t^{1/2}(\bth_{n\vp})}  \right],
\end{equation*}
the modulus of  (\ref{qit}) is bounded above by
$$
\underset{\substack{x \in \R \\ ||\bb|| < b}}{\sup} \, \frac{1}{n+1} \sum_{i=\floor*{n^k}}^{n} \left| I\left[ \frac{X_i}{v_i^{1/2}(\bth_{n\vp})} < x \right] - I\left[ \frac{X_i}{v_i^{1/2}(\bth_{n\vp})} \frac{v_i^{1/2}(\bth_{n\vp}) \h{v}_t^{1/2}(\bth_{n\vp})}{\h{v}_i^{1/2}(\bth_{n\vp}) v_t^{1/2}(\bth_{n\vp})} < x  \right] \right|. 
$$
Using $|I(A) - I(B)| \leq I({A \cap B^c}) +I({A^c \cap B})$, this is bounded above by  
\begin{equation*}
\underset{\substack{x \in \R \\ \bth \in \bTheta}}{\sup} \, \frac{1}{n+1} \sum_{i=\floor*{n^k}}^{n} I(\mathcal{A}_{i, x, \bth} ), 
\end{equation*}
where the set $\mathcal{A}_{i, x, \bth} $ is defined as 
\begin{align*}
\mathcal{A}_{i, x, \bth}  := & \left\lbrace \frac{X_i}{v_i^{1/2}(\bth)} < x,  \frac{X_i}{v_i^{1/2}(\bth)} \frac{v_i^{1/2}(\bth) \h{v}_t^{1/2}(\bth)}{\h{v}_i^{1/2}(\bth) v_t^{1/2}(\bth)} \ge x\right\rbrace \\
& \cup \left\lbrace \frac{X_i}{v_i^{1/2}(\bth)} \ge x,  \frac{X_i}{v_i^{1/2}(\bth)} \frac{v_i^{1/2}(\bth) \h{v}_t^{1/2}(\bth)}{\h{v}_i^{1/2}(\bth) v_t^{1/2}(\bth)} < x\right\rbrace.
\end{align*}
Therefore, it suffices to prove that 
$\sum_{i=\floor*{n^k}}^{n} I(\mathcal{A}_{i, x, \bth}  ) = o_{\rm P}(1)$ uniformly with respect to both $x$ and $\bth$. We show this with sets containing $\mathcal{A}_{i, x, \bth}$.

Recall that using (\ref{fact21}), $\h{v}_t(\bth) \geq c_0(\bth)> c$ for a positive constant $c$ and so
$$0 < v_t^{1/2}(\bth) - \h{v}_t^{1/2}(\bth) \leq \frac{\rho^t Z_0}{v_t^{1/2}(\bth) + \h{v}_t^{1/2}(\bth)} \leq \frac{\rho^t Z_0}{2c_0^{1/2}(\bth)}.$$
Now using the triangular inequality,
\begin{align}
&\quad \left| \frac{v_i^{1/2}(\bth) \h{v}_t^{1/2}(\bth)}{\h{v}_i^{1/2}(\bth) v_t^{1/2}(\bth)} - 1 \right| \nonumber \\
&\leq \left| \frac{\h{v}_t^{1/2}(\bth) \left(v_i^{1/2}(\bth) - \h{v}_i^{1/2}(\bth) \right)}{\h{v}_i^{1/2}(\bth) v_t^{1/2}(\bth)} \right| + \left|\frac{\h{v}_i^{1/2}(\bth) \left(\h{v}_t^{1/2}(\bth) - v_t^{1/2}(\bth) \right)}{\h{v}_i^{1/2}(\bth) v_t^{1/2}(\bth)} \right| \label{tri.v.vhat}.
\end{align}
Therefore (\ref{tri.v.vhat}) is bounded above by
$$\frac{\rho^i Z_0}{2c_0^{1/2}(\bth)} + \frac{\rho^t Z_0}{2c_0^{1/2}(\bth)}.$$
In view of (\ref{tri.v.vhat}), we get
$$\left| \frac{X_i}{v_i^{1/2}(\bth)} \frac{v_i^{1/2}(\bth) \h{v}_t^{1/2}(\bth)}{\h{v}_i^{1/2}(\bth) v_t^{1/2}(\bth)} - \frac{X_i}{v_i^{1/2}(\bth)} \right| \leq (\rho^i + \rho^t) Z_4 \left| \frac{X_i}{v_i^{1/2}(\bth)} \right|,$$
where $Z_4 = Z_0/(2 C^{1/2})$.

Therefore, $\mathcal{A}_{i, x, \bth} $  is a subset of 
\begin{align*}
\mathcal{B}_{i, x, \bth}  := &\left\lbrace \frac{X_i}{v_i^{1/2}(\bth)} < x, \frac{X_i}{v_i^{1/2}(\bth)} + (\rho^i + \rho^t) Z_4 \left| \frac{X_i}{v_i^{1/2}(\bth)} \right| \ge x \right\rbrace \\
& \cup \left\lbrace \frac{X_i}{v_i^{1/2}(\bth)} \ge x, \frac{X_i}{v_i^{1/2}(\bth)} - (\rho^i + \rho^t) Z_4 \left| \frac{X_i}{v_i^{1/2}(\bth)} \right| < x \right\rbrace \\
=& \left\lbrace \eta_i  <  x \frac{v_i^{1/2}(\bth)}{v_i^{1/2}(\bth_{0\vp})}, \eta_i  + (\rho^i + \rho^t) Z_4 |\eta_i| \ge x \frac{v_i^{1/2}(\bth)}{v_i^{1/2}(\bth_{0\vp})} \right\rbrace \\
& \cup \left\lbrace \eta_i  \ge x \frac{v_i^{1/2}(\bth)}{v_i^{1/2}(\bth_{0\vp})}, \eta_i  - (\rho^i + \rho^t)Z_4 |\eta_i| < x \frac{v_i^{1/2}(\bth)}{v_i^{1/2}(\bth_{0\vp})} \right\rbrace \\
=& \left\lbrace  x \frac{v_i^{1/2}(\bth)}{v_i^{1/2}(\bth_{0\vp})}-(\rho^i + \rho^t) Z_4 |\eta_i| \le \eta_i  <  x \frac{v_i^{1/2}(\bth)}{v_i^{1/2}(\bth_{0\vp})}\right\rbrace \\
& \cup \left\lbrace x \frac{v_i^{1/2}(\bth)}{v_i^{1/2}(\bth_{0\vp})} \le \eta_i< x \frac{v_i^{1/2}(\bth)}{v_i^{1/2}(\bth_{0\vp})} 
+(\rho^i + \rho^t)Z_4 |\eta_i| \right\rbrace.
\end{align*}
Consider r.v.s $X$ and $L \geq 0$ with $X$ independent of $\eta$. Then ${\mathrm P} [X<\eta<X+L] \leq \underset{y \in \R}{\sup} \{g(y)\} {\mathrm E} (L)$ where the p.d.f. $g$ is bounded. Consequently, 
$$
\underset{\substack{x \in \R \\ \bth \in \bTheta}}{\sup} \, {\mathrm E}\left[ \sum_{i=\floor*{n^k}}^{n} I(\mathcal{A}_{i, x, \bth}  ) \right] \leq \underset{y \in \R}{\sup} \, g(y) \sum_{i=\floor*{n^k}}^{n} {\mathrm E}\left\lbrace  (\rho^i + \rho^t) Z_4 |\eta_i| \right\rbrace.
$$
Notice that since $\floor*{n^k} \leq t \leq n$,
$$\sum_{i=\floor*{n^k}}^{n} {\mathrm E}\left\lbrace \rho^t Z_4 |\eta_i| \right\rbrace \leq n \rho^{\floor*{n^k}} {\mathrm E}(Z_4) {\mathrm E} |\eta| = o(1)$$
due to $0<\rho<1$ and ${\mathrm E} |\eta|<\infty$. Hence, $\sum_{i=\floor*{n^k}}^{n} I(\mathcal{A}_{i, x, \bth})$ converges in mean to zero uniformly with respect to both $x$ and $\bth$, which entails $\sum_{i=\floor*{n^k}}^{n} I(\mathcal{A}_{i, x, \bth}  ) = o_{\rm P}(1)$ uniformly.



\color{black}

\end{proof}

With all the results above, we can easily prove Proposition 
\ref{Prop.asy.linear} and the asymptotic result for the R-estimator as follows. \\

\textbf{Proof of Proposition 
\ref{Prop.asy.linear}.}
\begin{proof}

Combining (\ref{T1n}), (\ref{rank2}), (\ref{Tn21M}) and (\ref{rank3}), we get
\begin{equation}\label{asy.linear.app}
\h{{\mbf R}}_n(\bth_{n\vp}) - \M(\bth_{0\vp}) \bb - \Q_n(\bth_{0\vp}) - \N_n(\bth_{0\vp}) -    \J(\bth_{0\vp}) \bb/2  = u_{\rm P}(1),
\end{equation}
which, by letting $\bb = \0$, entails
\begin{equation*}
\h{{\mbf R}}_n(\bth_{0\vp}) = \Q_n(\bth_{0\vp}) + \N_n(\bth_{0\vp}) +  u_{\rm P}(1).
\end{equation*}
Hence, (
\ref{Prop.asy.lin1}) follows by recalling that $\M(\bth_{0\vp}) = \J(\bth_{0\vp}) \rho(\vp)/2$. 

The proof of (
\ref{Prop.asy.lin2}) follows directly from  (\ref{Nn.Norm}) and (\ref{muGGn1}).

\end{proof}

\textbf{Proof of Theorem 
\ref{thm.asy}.}

\begin{proof}
From the definition of $\h{\bth}_{n\vp}$ in 
(\ref{onestep.def}), (
\ref{Prop.asy.lin1}) and (
\ref{Prop.asy.lin2}) in Proposition 
\ref{Prop.asy.linear}, consistency of $\h{\bUpsilon}_n$
and the asymptotic discreteness of $ \bar{\bth}_n$ (which allows us to treat $n^{1/2}( \bar{\bth}^{(n)} -  \bth )$ as if it were a bounded constant: see   Lemma~4.4  in Kreiss~(1987)), we have
\begin{align*}
& n^{1/2}(\h{\bth}_{n\vp} - \bth_{0\vp}) \\
=& n^{1/2} \left\lbrace  \bar{\bth}_n - n^{-1/2} \left( \h{\bUpsilon}_n \right)^{-1} \h{\mbf R}_n ( \bar{\bth}_n) - \bth_{0\vp} \right\rbrace \\
=& n^{1/2} \left\lbrace  \bar{\bth}_n - n^{-1/2} \left( \h{\bUpsilon}_n \right)^{-1} \left[ \h{\mbf R}_n (\bth_{0\vp}) + (1/2 + \rho(\vp)/2 )\J(\bth_{0\vp}) n^{1/2} ( \bar{\bth}_n - \bth_{0\vp} )\right]- \bth_{0\vp} \right\rbrace + o_{\rm P}(1)\\
=& n^{1/2} \left\lbrace  \bar{\bth}_n - n^{-1/2} (1/2 + \rho(\vp)/2 )^{-1} (\J(\bth_{0\vp}))^{-1} \h{\mbf R}_n (\bth_{0\vp}) - ( \bar{\bth}_n - \bth_{0\vp} ) - \bth_{0\vp} \right\rbrace + o_{\rm P}(1)\\
=&  -(1/2 + \rho(\vp)/2 )^{-1} (\J(\bth_{0\vp}))^{-1} \h{\mbf R}_n(\bth_{0\vp}) + o_{\rm P}(1).
\end{align*}
In view of (
\ref{Prop.asy.lin2}), we have
\begin{equation*}
n^{1/2}(\h{\bth}_{n\vp} - \bth_{0\vp}) = -(1/2 + \rho(\vp)/2 )^{-1} (\J(\bth_{0\vp}))^{-1}  (\Q_n(\bth_{0\vp}) + \N_n(\bth_{0\vp})) + o_{\rm P}(1).
\end{equation*}

Now, it remains to obtain the asymptotic covariance matrix of $\sqrt{n} \left(\h{\bth}_{n\vp} - \bth_{0\vp}\right)$. 
Recall (\ref{Nn.Norm}) and (\ref{muGGn1}). Since the asymptotic covariance matrices of $\Q_n(\bth_{0\vp})$ and $\N_n(\bth_{0\vp})$ have been derived, it remains to obtain the covariance matrix $\text{Cov}(\Q_n(\bth_{0\vp}), \N_n(\bth_{0\vp}))$. Note that ${\mathrm E}[\Q_n(\bth_{0\vp}) \N_n^\prime(\bth_{0\vp})]$ equals
\begin{align}
&{\mathrm E}\left\lbrace \left[ \int_{0}^{1}  n^{-1/2} \s \frac{\d{\mbf v}_t(\bth_{0\vp})}{v_t(\bth_{0\vp})} \left[ \mu(G^{-1}(u)) -  \mu(\tilde{G}_n^{-1}(u))\right]  
d \varphi(u) \right] \right. \nonumber \\
&\qquad \times \left.  \left[ n^{-1/2} \s \frac{\d{\mbf v}_t^\prime(\bth_{0\vp})}{v_t(\bth_{0\vp})}\left[ 1-\eta_t \vp\left[ G(\eta_t) \right] \right] \right]\right\rbrace. \label{QNcov1}
\end{align}

Using (\ref{Qn1}) and $n^{-1} \s \d{\mbf v}_t(\bth_{0\vp})/v_t(\bth_{0\vp}) \rightarrow {\mathrm E}(\d{\mbf v}_1(\bth_{0\vp})/v_1(\bth_{0\vp}))$, as $n\rightarrow \infty$, (\ref{QNcov1}) has the same limit as
\begin{align}
&{\mathrm E}\left( \frac{\d{\mbf v}_1(\bth_{0\vp})}{v_1(\bth_{0\vp})}  \right) \nonumber  \\
& \times \underset{n\rightarrow \infty}{\lim} {\mathrm E}\left\lbrace \int_{0}^{1}  G^{-1}(u) \left\lbrace \frac{1}{n}   \sum_{i=1}^n \sum_{j=1}^n \left[I\{\eta_i \leq G^{-1}(u) \} - u\right]  \frac{\d{\mbf v}_j^\prime(\bth_{0\vp})}{v_j(\bth_{0\vp})}\left[ 1-\eta_j \vp\left[ G(\eta_j) \right] \right] \right\rbrace	d \varphi(u) \right\rbrace \nonumber  \\
&={\mathrm E}\left( \frac{\d{\mbf v}_1(\bth_{0\vp})}{v_1(\bth_{0\vp})}  \right) \nonumber  \\
& \quad \times \underset{n\rightarrow \infty}{\lim}   {\mathrm E}\left\lbrace \int_{0}^{1}  G^{-1}(u) \left\lbrace \frac{1}{n}   \sum_{i=1}^n \left[I\{\eta_i \leq G^{-1}(u) \} - u\right]  \frac{\d{\mbf v}_i^\prime(\bth_{0\vp})}{v_i(\bth_{0\vp})}\left[ 1-\eta_i \vp\left[ G(\eta_i) \right] \right] \right\rbrace d \varphi(u) \right\rbrace \nonumber \\
&={\mathrm E}\left( \frac{\d{\mbf v}_1(\bth_{0\vp})}{v_1(\bth_{0\vp})}  \right) {\mathrm E}\left( \frac{\d{\mbf v}_1^\prime(\bth_{0\vp})}{v_1(\bth_{0\vp})}  \right)    \int_{0}^{1}  G^{-1}(u) {\mathrm E}\left\lbrace \left[I\{\eta_1 \leq G^{-1}(u) \} - u\right]  \left[ 1-\eta_1 \vp\left[ G(\eta_1) \right] \right] \right\rbrace d \varphi(u) \nonumber \\
&={\mathrm E}\left( \frac{\d{\mbf v}_1(\bth_{0\vp})}{v_1(\bth_{0\vp})}  \right) {\mathrm E}\left( \frac{\d{\mbf v}_1^\prime(\bth_{0\vp})}{v_1(\bth_{0\vp})}  \right)    \int_{0}^{1}  G^{-1}(u) {\mathrm E}\left\lbrace I\{\eta_1 \leq G^{-1}(u) \}   \left[ 1-\eta_1 \vp\left[ G(\eta_1) \right] \right] \right\rbrace d \varphi(u), \nonumber 
\end{align}
where the first equality is due to independence of $\eta_i$ and $\eta_j$ for $i\neq j$, independence of $v_j$ and $\eta_j$, and Assumption (A1). The second equality is due to independence of $v_i$ and $\eta_i$. The last equality is due to Assumption (A1).

Recall the definition of  $\lambda(\vp)$ in (
\ref{lam.def}), which can also be written as
\begin{equation*}
\lambda(\vp)=  \int_{0}^{1}  G^{-1}(u) {\mathrm E}\left\lbrace I\{\eta_1 \leq G^{-1}(u) \}   \left[ 1-\eta_1 \vp\left[ G(\eta_1) \right] \right] \right\rbrace d \varphi(u).
\end{equation*}
We then have
\begin{equation*}
\underset{n\rightarrow \infty}{\lim} \text{Cov}(\Q_n(\bth_{0\vp}), \N_n(\bth_{0\vp})) = {\mathrm E}\left( \frac{\d{\mbf v}_1(\bth_{0\vp})}{v_1(\bth_{0\vp})}  \right) {\mathrm E}\left( \frac{\d{\mbf v}_1^\prime(\bth_{0\vp})}{v_1(\bth_{0\vp})}  \right) \lambda(\vp).
\end{equation*}

Hence, by recalling (\ref{muGGn1}) and in view of (
\ref{asy}), the asymptotic covariance matrix of $\sqrt{n} \left(\h{\bth}_{n\vp} - \bth_{0\vp}\right)$ is 
\begin{align*}
&\quad (\J(\bth_{0\vp}))^{-1} \frac{{\mathrm E}\left( \frac{\d{\mbf v}_1(\bth_{0\vp})}{v_1(\bth_{0\vp})}  \right) {\mathrm E}\left( \frac{\d{\mbf v}_1^\prime(\bth_{0\vp})}{v_1(\bth_{0\vp})}  \right) \text{Var}(Z) + \J(\bth_{0\vp}) \sigma^2(\vp) + 2 {\mathrm E}\left( \frac{\d{\mbf v}_1(\bth_{0\vp})}{v_1(\bth_{0\vp})}  \right) {\mathrm E}\left( \frac{\d{\mbf v}_1^\prime(\bth_{0\vp})}{v_1(\bth_{0\vp})}  \right) \lambda(\vp)}{(1/2 + \rho(\vp)/2)^2}   (\J(\bth_{0\vp}))^{-1} \\
&=  (\J(\bth_{0\vp}))^{-1} \frac{\left[ 4\gamma(\vp)  + 8  \lambda(\vp) \right]  {\mathrm E}\left( \frac{\d{\mbf v}_1(\bth_{0\vp})}{v_1(\bth_{0\vp})}  \right) {\mathrm E}\left( \frac{\d{\mbf v}_1^\prime(\bth_{0\vp})}{v_1(\bth_{0\vp})}  \right) + 4 \sigma^2(\vp) \J(\bth_{0\vp}) }{(1 + \rho(\vp))^2}  (\J(\bth_{0\vp}))^{-1}.
\end{align*}

\end{proof}

\noindent \textbf{Proof of Theorem 
\ref{thm.asybth0}.}

\begin{proof}
From (\ref{hatcvp}), 
$$
\hat{c}_{n\vp}= \left(1 -  \sum_{j=1}^q  \hat{\beta}_{nj} \right)^{-1} \left( \hat{\omega}_{n\vp}/\overline{X_n^2}  + \sum_{i = 1}^p  \hat{\alpha}_{n\vp i} \right)
= \left(1 - \sum_{j=1}^q  \beta_{0j} \right)^{-1} \left( c_{\vp} \omega_0/{\rm E}(X_t^2)  +c_{\vp} \sum_{i = 1}^p  \alpha_{0i} \right)+o_{\rm P}(1). 
$$
Substituting ${\rm E}(X_t^2)=\omega_0/\{1- \sum_{i = 1}^p \alpha_{0i} - \sum_{j=1}^q \beta_{0j}\}$, this is equal to $c_{\vp}$.  

To prove \eqref{asybth0}, we define a $(1+p+q) \times  (1+p+q)$  diagonal matrix
$$\A_{\vp} := \text{diag}(\underbrace{c_\vp^{-1}, ..., c_\vp^{-1}}_{1+p} \underbrace{1, ..., 1}_{q}).$$ 
Then we have $\bth_0 = \A_{\vp} \bth_{0\vp}$ and ${\bth}_n =  \A_{\vp} {\bth}_{n\vp}.$
Using the forms of $\{c_j (\bth); j \ge 0\}$ in (3.1) of Berkes et al.~(2003),  we have
$$v_t(\bth_{0\vp}) = c_\vp v_t(\bth_0) \quad \text{and} \quad \d{\mbf v}_t(\bth_{0\vp}) = c_\vp \A_{\vp} \d{\mbf v}_t(\bth_0)$$ 
which entail
$$\J(\bth_{0\vp}) = \A_{\vp} \J(\bth_0)  \A_{\vp}, \quad \Q_n(\bth_{0\vp}) =  \A_{\vp}  \Q_n(\bth_0), \quad \N_n(\bth_{0\vp}) = \A_{\vp}  \N_n(\bth_0).$$
Hence \eqref{asybth0} follows from \eqref{asy} and consistency of $\hat{c}_{n\vp}$.

%
%

\end{proof}

\section{R-estimators in GJR models and applications}
\label{sec.GJR}
Glosten et al.~(1993) proposed the GJR~($p, q$) model for the asymmetric volatility observed in many financial dataset
 exhibiting asymmetry property. The GJR~($p, q$) model is defined by 
$$X_t = \sigma_t \epsilon_t$$
where
$$\sigma_t^2 = \omega_0 + \sum_{i = 1}^p \left[\alpha_{0i} + \gamma_{0i} I(X_{t-i} <0)\right] X_{t-i}^2 + \sum_{j = 1}^q  \beta_{0j} \sigma_{t-j}^2, \quad t \in \ZZ,$$
with $\omega_0, \alpha_{0i}, \gamma_{0i}, \beta_{0j} > 0, \forall i, j.$ 
Since $\sigma_t^2$ is linear in parameters,   
we define the R-estimators for the GJR model using the same rank-based central sequence as in (
\ref{def.R}). See also Iqbal and Mukherjee~(2010) for the extension of M-estimators from the GARCH model to the GJR model.
We do not prove any asymptotic theory for the R-estimators of the GJR model but present here empirical analysis using the same algorithm as in 
Algorithm~\ref{algRestGARCH}
to compute the R-estimators. The following extensive simulation study, similar to the GARCH case, demonstrates the superior performance of the R-estimators compared to the QMLE that is often used in this model. We also carry out simulation with increasing sample sizes to show the consistency of the R-estimators.
Three types of R-estimators and the QMLE are compared below under various error distributions. We run simulations with the sample size $n = 1000$, number of replications $R= 500$ and true parameter
$$\bth_0 = (3.45\times 10^{-4}, 0.0658, 0.0843, 0.8182)^\prime,$$
which is motivated by the estimate in Tsay~(2010) for the IBM stock monthly returns from 1926 to 2003. The estimates of the bias and MSE of the QMLE and R-estimators and those of the ARE of the R-estimators wrt the QMLE are reported in Table \ref{GJR.simulation}. 

We remark that the results are consistent with those in Table 
\ref{simulation1}: the vdW score still dominates the QMLE uniformly; the optimal scores under the DE and logistic distributions are also the sign and Wilcoxon, respectively. It is worth noting that under $t(3)$ distribution, the R-estimators are much more efficient than the QMLE for the parameter $\gamma.$

\begin{table}[!htbp]
\caption{The estimates of the bias, MSE and ARE of  the R-estimators (sign, Wilcoxon and vdW) and the QMLE for the GJR~(1, 1) model under various error distributions (sample size $n = 1000$; $R= 500$ replications).}\vspace{-3mm}
\label{GJR.simulation} 
\begin{center}
\scriptsize  
\begin{tabular}{c c c c c c c c c c}\hline 
&\multicolumn{4}{c}{\textbf{Bias}} &&\multicolumn{4}{c}{\textbf{MSE and ARE}} \\ \cline{2-5} \cline{7-10} & $\omega$ &$\alpha$& $\gamma$ & $\beta$&  &$\omega$ &$\alpha$ & $\gamma$ &$\beta$ \\ \hline
\textbf{Normal} &&&&&&&&&  \\ 
QMLE	 & 8.70$\times 10^{-5}$	&-2.03$\times 10^{-3}$&	7.47$\times 10^{-3}$&	-2.03$\times 10^{-2}$ && 4.17$\times 10^{-8}$& 7.63$\times 10^{-4}$&	1.53$\times 10^{-3}$	&3.80$\times 10^{-3}$ \\ 
Sign&  9.01$\times 10^{-5}$&	-1.20$\times 10^{-3}$&	8.43$\times 10^{-3}$&	-2.00$\times 10^{-2}$&&		4.73$\times 10^{-8}$&	9.40$\times 10^{-4}$&	1.92$\times 10^{-3}$&	4.30$\times 10^{-3}$\\
&&&&&&(0.88)&	(0.81)&	(0.80) &	(0.88)\\
Wilcoxon&9.35$\times 10^{-5}$&	-8.15$\times 10^{-4}$&	8.65$\times 10^{-3}$&	-2.00$\times 10^{-2}$&&		4.96$\times 10^{-8}$&	8.72$\times 10^{-4}$&	1.76$\times 10^{-3}$&	4.24$\times 10^{-3}$	 \\ 
&&&&&&(0.84)&	(0.87)&	(0.87) &	(0.90)\\
vdW	&8.72$\times 10^{-5}$&	-1.61$\times 10^{-3}$&	7.59$\times 10^{-3}$&	-2.01$\times 10^{-2}$&&		4.24$\times 10^{-8}$&	7.72$\times 10^{-4}$&	1.56$\times 10^{-3}$&	3.82$\times 10^{-3}$ \\
&&&&&&(0.98) &	(0.99) & 	(0.98) &	(0.99) \\ \hline

\textbf{DE} &&&&&&& && \\ 
QMLE	& 7.07$\times 10^{-5}$&	-1.28$\times 10^{-3}$&	1.13$\times 10^{-2}$&	-1.87$\times 10^{-2}$&&		4.66$\times 10^{-8}$&	1.23$\times 10^{-3}$&	3.19$\times 10^{-3}$&	4.72$\times 10^{-3}$ \\
Sign&	4.91$\times 10^{-5}$&	-2.26$\times 10^{-3}$&	6.99$\times 10^{-3}$&	-1.62$\times 10^{-2}$&&		3.44$\times 10^{-8}$&	9.58$\times 10^{-4}$	&2.36$\times 10^{-3}$&	4.10$\times 10^{-3}$ \\
&&&&&&(1.35) &	(1.28) &	(1.35) &	(1.15) \\
Wilcoxon&	5.06$\times 10^{-5}$&	-2.23$\times 10^{-3}$&	7.25$\times 10^{-3}$&	-1.61$\times 10^{-2}$	&&	3.45$\times 10^{-8}$&	9.74$\times 10^{-4}$&	2.41$\times 10^{-3}$	& 4.06$\times 10^{-3}$ \\
&&&&&&(1.35) &	(1.26) &	(1.32) &	(1.16) \\
vdW	&4.98$\times 10^{-5}$&	-3.35$\times 10^{-3}$&	6.73$\times 10^{-3}$&	-1.71$\times 10^{-2}$&&		3.54$\times 10^{-8}$&	1.02$\times 10^{-3}$&	2.50$\times 10^{-3}$&	4.23$\times 10^{-3}$ \\
&&&&&&(1.32) &(1.20) &	(1.27) &	(1.12)  \\ \hline

\textbf{Logistic} &&&&&&& \\ 
QMLE& 8.01$\times 10^{-5}$&	7.88$\times 10^{-4}$&	8.85$\times 10^{-3}$&	-1.62$\times 10^{-2}$&&		3.86$\times 10^{-8}$&	1.07$\times 10^{-3}$&	2.40$\times 10^{-3}$&	3.46$\times 10^{-3}$\\
Sign&6.06$\times 10^{-5}$&	-2.59$\times 10^{-4}$&	5.54$\times 10^{-3}$&	-1.49$\times 10^{-2}$&&		3.26$\times 10^{-8}$&	8.97$\times 10^{-4}$&	1.96$\times 10^{-3}$&	3.36$\times 10^{-3}$	\\
&&&&&&(1.18)&	(1.19)&	(1.23)&	(1.03)\\
Wilcoxon&5.97$\times 10^{-5}$&	-5.22$\times 10^{-4}$	&5.38$\times 10^{-3}$&	-1.44$\times 10^{-2}$	&&	3.07$\times 10^{-8}$&	8.76$\times 10^{-4}$&	1.93$\times 10^{-3}$&	3.18$\times 10^{-3}$	\\
&&&&&&(1.26)&	(1.22)&	(1.24)&	(1.09) \\
vdW	& 6.27$\times 10^{-5}$	&-1.18$\times 10^{-3}$&	5.18$\times 10^{-3}$&	-1.56$\times 10^{-2}$	&&	3.25$\times 10^{-8}$&	9.26$\times 10^{-4}$&	2.04$\times 10^{-3}$&	3.35$\times 10^{-3}$ \\
&&&&&&(1.19)&	(1.15)&	(1.17)&	(1.03) \\  \hline

${\boldsymbol {t(3)}}$ &&&&&&&\\ 
QMLE&9.91$\times 10^{-5}$&	-1.00$\times 10^{-3}$&	9.21$\times 10^{-2}$&	-6.45$\times 10^{-2}$	&&	1.14$\times 10^{-7}$&	4.38$\times 10^{-3}$&	1.18$\times 10^{-1}$&	2.31$\times 10^{-2}$	\\
Sign&	5.68$\times 10^{-5}$&	4.23$\times 10^{-4}$&	2.71$\times 10^{-2}$&	-2.77$\times 10^{-2}$&&		3.78$\times 10^{-8}$&	1.35$\times 10^{-3}$&	4.69$\times 10^{-3}$&	6.33$\times 10^{-3}$ \\
&&&&&&(3.02)&	(3.24)&	(25.19)&	(3.65) \\
Wilcoxon&5.69$\times 10^{-5}$	&3.94$\times 10^{-5}$&	2.74$\times 10^{-2}$&	-2.85$\times 10^{-2}$&&		3.93$\times 10^{-8}$&	1.40$\times 10^{-3}$&	4.80$\times 10^{-3}$&	6.62$\times 10^{-3}$	\\
&&&&&&(2.91)&	(3.12)&	(24.57)&	(3.49)	\\
vdW&6.12$\times 10^{-5}$&	-1.61$\times 10^{-3}$&	3.43$\times 10^{-2}$&	-3.71$\times 10^{-2}$	&&	5.13$\times 10^{-8}$&	1.93$\times 10^{-3}$&	7.45$\times 10^{-3}$&	9.58$\times 10^{-3}$	\\	
&&&&&&(2.23)&	(2.27)&	(15.84)&	(2.41)	\\ \hline   
\end{tabular}
\end{center}
\end{table}

\textbf{Simulation under different sample size.} We next investigate behaviour of the R-estimators by carrying out simulations with different sample sizes. The number of replications and true parameter are the same as those used for Table \ref{GJR.simulation} and the error distribution is normal. The estimates of the bias and MSE of the R-estimators for the GJR~($1, 1$) model are shown in 
Table \ref{GJR.simulation2}. In general, for all R-estimators, both the bias and MSE decrease when the sample size increases
from $n = 500$ to $n = 5000$. This tends to reflect that the R-estimators are consistent estimators of $\bth_{0}$ for the GJR~($1, 1$) model.

\begin{table}
\caption{The bias, MSE of the R-estimators (sign, Wilcoxon and vdW) for the GJR~(1, 1) model under normal error distributions with different sample sizes ($R = 500$ replications).} \vspace{-3mm}
\label{GJR.simulation2}
\begin{center}
\scriptsize  
\begin{tabular}{c c c c c c c c c c}\hline 
&\multicolumn{4}{c}{\textbf{Bias}} &&\multicolumn{4}{c}{\textbf{MSE}} \\ \cline{2-5} \cline{7-10} & $\omega$ &$\alpha$& $\gamma$ & $\beta$&  &$\omega$ &$\alpha$ & $\gamma$ &$\beta$ \\ \hline
\textbf{Sign} &&&&&&&&&  \\ 
$n = 500$	 &1.78$\times 10^{-4}$	&2.42$\times 10^{-3}$&	1.88$\times 10^{-2}$&	-4.84$\times 10^{-2}$&&		1.43$\times 10^{-7}$&	1.64$\times 10^{-3}$&	3.72$\times 10^{-3}$&	1.33$\times 10^{-2}$ \\ 
$n = 1000$&9.01$\times 10^{-5}$	&-1.20$\times 10^{-3}$&	8.43$\times 10^{-3}$&	-2.00$\times 10^{-2}$&&		4.73$\times 10^{-8}$&	9.40$\times 10^{-4}$&	1.92$\times 10^{-3}$&	4.30$\times 10^{-3}$\\
$n = 3000$& 2.76$\times 10^{-5}$&	-6.82$\times 10^{-4}$&	1.97$\times 10^{-3}$&	-4.47$\times 10^{-3}$&&		8.75$\times 10^{-9}$&	3.05$\times 10^{-4}$&	6.11$\times 10^{-4}$&	8.73$\times 10^{-4}$\\ 
$n = 5000$& 2.07$\times 10^{-5}$&	-4.43$\times 10^{-4}$&	2.05$\times 10^{-3}$&	-3.43$\times 10^{-3}$&&		4.57$\times 10^{-9}$&	1.70$\times 10^{-4}$&	3.68$\times 10^{-4}$&	4.62$\times 10^{-4}$ \\ \hline

\textbf{Wilcoxon} &&&&&&&&&  \\ 
$n = 500$	 & 1.77$\times 10^{-4}$	&2.52$\times 10^{-3}$&	1.88$\times 10^{-2}$&	-4.75$\times 10^{-2}$&&		1.37$\times 10^{-7}$&	1.52$\times 10^{-3}$&	3.45$\times 10^{-3}$&	1.26$\times 10^{-2}$\\ 
$n = 1000$& 9.35$\times 10^{-5}$&	-8.15$\times 10^{-4}$&	8.65$\times 10^{-3}$&	-2.00$\times 10^{-2}$&&		4.96$\times 10^{-8}$&	8.72$\times 10^{-4}$&	1.76$\times 10^{-3}$&	4.24$\times 10^{-3}$\\
$n = 3000$& 3.01$\times 10^{-5}$&	-2.94$\times 10^{-5}$&	2.68$\times 10^{-3}$&	-4.30$\times 10^{-3}$&&		8.18$\times 10^{-9}$&	2.82$\times 10^{-4}$&	5.60$\times 10^{-4}$&	7.85$\times 10^{-4}$\\ 
$n = 5000$& 2.45$\times 10^{-5}$&	1.52$\times 10^{-4}$&	2.82$\times 10^{-3}$&	-3.54$\times 10^{-3}$&&		4.50$\times 10^{-9}$&	1.63$\times 10^{-4}$&	3.55$\times 10^{-4}$&	4.29$\times 10^{-4}$\\ \hline 

\textbf{vdW} &&&&&&&&&  \\ 
$n = 500$	 &1.67$\times 10^{-4}$&	1.42$\times 10^{-3}$&	1.61$\times 10^{-2}$&	-4.84$\times 10^{-2}$&&		1.31$\times 10^{-7}$	&1.44$\times 10^{-3}$&	3.03$\times 10^{-3}$&	1.27$\times 10^{-2}$ \\ 
$n = 1000$& 8.72$\times 10^{-5}$&	-1.61$\times 10^{-3}$&	7.59$\times 10^{-3}$&	-2.01$\times 10^{-2}$&&		4.24$\times 10^{-8}$&	7.72$\times 10^{-4}$&	1.56$\times 10^{-3}$	&3.82$\times 10^{-3}$\\
$n = 3000$& 2.88$\times 10^{-5}$&	-1.72$\times 10^{-4}$&	1.60$\times 10^{-3}$&	-4.59$\times 10^{-3}$&&		7.42$\times 10^{-9}$&	2.61$\times 10^{-4}$&	4.90$\times 10^{-4}$&	7.19$\times 10^{-4}$\\ 
$n = 5000$& 2.42$\times 10^{-5}$&	9.50$\times 10^{-6}$&	2.42$\times 10^{-3}$&	-4.02$\times 10^{-3}$&&		4.38$\times 10^{-9}$&	1.49$\times 10^{-4}$&	3.28$\times 10^{-4}$&	4.26$\times 10^{-4}$\\ \hline
\end{tabular}
\end{center}
\end{table}

\end{document}